\documentclass[11pt]{amsart}
\usepackage[english]{babel}
\usepackage[latin1]{inputenc}

\usepackage[a4paper]{geometry}
\usepackage{amsmath,amsthm}
\usepackage{amssymb,amsfonts}
\usepackage{hyperref}
\usepackage{tikz}
\usepackage{microtype}

\usepackage{stmaryrd}
\DeclareSymbolFont{bbold}{U}{bbold}{m}{n}
\DeclareSymbolFontAlphabet{\mathbbm}{bbold}

\usepackage{enumerate}
	\usetikzlibrary{shapes}

\usepackage{amsxtra}
\usepackage{mathrsfs}
\usepackage{graphicx}
\usepackage{mathabx}
\theoremstyle{plain}
	\newtheorem{theorem}{Theorem}[section]
	
	\newtheorem{lemma}[theorem]{Lemma}
	\newtheorem{corollary}[theorem]{Corollary}

\theoremstyle{definition}
	\newtheorem{definition}[theorem]{Definition}
	\newtheorem{notation}[theorem]{Notation}
	\newtheorem{example}[theorem]{Example}
	\newtheorem{notation*}{Notation}

\theoremstyle{remark}
	\newtheorem{remark}[theorem]{Remark}

\newcommand{\ZFC}{\ensuremath{\mathsf{ZFC}}}
\newcommand{\ZF}{\ensuremath{\mathsf{ZF}}}

\DeclareMathOperator{\dom}{dom}

\DeclareMathOperator{\Cod}{\mathrm{Cod}}
\DeclareMathOperator{\Coll}{Coll}

\DeclareMathOperator{\UB}{\mathsf{UB}}

\DeclareMathOperator{\ST}{\mathsf{ST}}

\newcommand{\M}{\mathcal{M}}
\newcommand{\LL}{\mathcal{L}}
\newcommand{\bool}[1]{\mathsf{#1}}

\newcommand{\pow}[1]{\mathcal{P}\left(#1\right)}

\newcommand{\cp}[1]{\left( #1 \right)}
\newcommand{\qp}[1]{\left[ #1 \right]}
\newcommand{\Qp}[1]{\left\llbracket #1 \right\rrbracket}
\newcommand{\ap}[1]{\langle #1 \rangle}
\newcommand{\bp}[1]{\left\lbrace #1 \right\rbrace}

\theoremstyle{definition}

\newcommand{\UnderTilde}[1]{{\setbox1=\hbox{$#1$}\baselineskip=0pt\vtop{\hbox{$#1$}\hbox to\wd1{\hfil$\sim$\hfil}}}{}}

\begin{document}

\title{Second order arithmetic as the model companion 
of set theory}
\author{Giorgio Venturi and Matteo Viale}

\begin{abstract}
This is an introductory paper to a series of results linking generic absoluteness results for 
second and third order number theory to the model theoretic notion of model companionship.
Specifically we develop here a general framework linking 
Woodin's generic absoluteness results for second order number theory and the theory of universally Baire sets
to model companionship and
show that (with the required care in details) 
a $\Pi_2$-property formalized in an appropriate language for second order number theory 
is forcible from some 
$T\supseteq\ZFC+$\emph{large cardinals}
if and only if it is consistent with the universal fragment of $T$
if and only if it is realized in the model companion of $T$.

In particular we show that the first order theory of $H_{\omega_1}$ is the model companion of the first order theory of the universe of sets 
assuming the existence of class many Woodin cardinals, and working in a signature 
with predicates for $\Delta_0$-properties and for all universally Baire sets of reals.

We will extend these results also to the theory of $H_{\aleph_2}$ in a follow up of this paper.
\end{abstract}

\maketitle

\section*{Introduction}
This paper outlines
a deep connection between
two important threads of mathematical logic: the notion of model companionship, 
a central concept in model theory due to Robinson, and the notion of 
generic absoluteness, which plays a fundamental role in the current meta-mathematical investigations 
of set theory. 

In order to unveil this connection, we proceed as follows:
we enrich the first order language in which to formalize set theory by  predicates 
whose meaning is as ``clear'' as that of the $\in$-relation, 
specifically we add predicates for $\Delta_0$-formulae
and predicates for universally Baire sets of reals\footnote{It is a standard result of set theory that
$\Delta_0$-formulae define absolute properties for transitive models of $\ZFC$. On the other hand the notion of universal Baireness captures exactly those sets of reals whose first order properties cannot be changed by means of forcing (for example all Borel sets of reals are universally Baire). Therefore these predicates have a meaning which is clear across the different models of set theory. See also the last part of this introduction.}
%
%We do not expand further on this matter here, we just remark that: on the one hand a fine classification of which sets of reals are universally Baire and which are not would bring us into rather delicate grounds; on the other hand the results of this paper are based on the closure under first order definability of the class of universally Baire sets (i.e. closure under projections, finite intersections, finite unions, complementation),  which is the case if we assume the existence of class many Woodin cardinals.}.
In this extended language we are able to apply Robinson's notions of model completeness and model companionship
to argue that, assuming large cardinals, the first order theory of $H_{\omega_1}$ (the family of all hereditarily countable sets) is model complete and is the model companion of the first theory of $V$ (the universe of all sets). 

The study of model companionship goes back to the work of Abraham Robinson
from the period 1950--1957 \cite{MacCompl}, and gives an abstract model-theoretic characterization of key closure properties of algebraically closed fields.
Robinson introduced the notion of model completeness to characterize the closure properties of algebraically closed fields, and the notion of model companionship to describe the relation existing between these fields and the commutative rings without zero-divisors. 
Robinson then showed how to extend these notions and results 
to a variety of other classes of first order structures.
On the other hand, generic absoluteness characterizes exactly those set theoretic properties whose truth value
cannot be changed by means of forcing. 

In \cite{VenturiRobinson} the first author found the first indication of a strict connection existing between 
these two apparently unrelated concepts. In this paper we will enlighten this connection much further.

Recall that a first order theory $T$ in a signature $\tau$ is model-complete if whenever $\mathcal{M}\sqsubseteq\mathcal{N}$ are models of $T$ with one a substructure of the other, we get that
$\mathcal{M}\prec\mathcal{N}$; i.e. being a substructure amounts to be an elementary substructure.

The theory of algebraically closed fields has this property, as it occurs for all theories admitting 
quantifier-elimination. However, it is also the case that many natural 
theories not admitting quantifier-elimination are model-complete. 
Robinson regarded model-completeness as a strong indication of tameness for a first order theory.

A weak point of this notion is that being model complete
is very sensitive to the signature in which a theory is formalized, to the extent that any 
theory $T$ in a signature $\tau$ has a conservative extension  $T'$ in a signature $\tau'$, which admits quantifier elimination (it suffices to add symbols and axioms for Skolem functions
to $\tau$ and $T$, \cite[Thm. 5.1.8]{TENZIE}). 
In particular we can always extend a first order language $\tau$ to a language $\tau'$ so to make a 
$\tau$-theory $T$
model-complete with respect to $\tau'$.
However if model-completeness of $T$ is shown with respect to a ``natural'' language in which $T$ can be formalized, then it brings many useful informations on the combinatorial-algebraic properties of models of $T$.

Recall also that for a first order signature $\tau$, 
a $\tau$-theory $T$ is the model companion of a $\tau$-theory $S$ if 
$T$ is model complete and $S$ and $T$ are mutually consistent: i.e.,  every model of $T$ can be embedded in a model of $S$ and
conversely. 

The notion of model companionship is much more robust than model completeness.
Consider the category $\mathcal{K}_{S,\tau}$ whose objects are the $\tau$-models of $S$ and whose arrows are given by the $\tau$-morphisms: knowing that $S$ admits a model companion gives non-trivial information on this category, and the change of signature from $\tau$ to 
$\tau'$ could bring our focus on something which is poorly related with 
$\mathcal{K}_{S,\tau}$.
Assume we enlarge the signature from $\tau$ to $\tau'$ so that in signature $\tau'$ 
$S$ has quantifier elimination, this has strong consequences on the substructure relation, hence it could be the case that for models $\mathcal{M}$, $\mathcal{N}$ of $S$
$\mathcal{M}$ is a $\tau$-substructure of $\mathcal{N}$
but it is not anymore a $\tau'$-substructure of $\mathcal{N}$. 
In particular the category $\mathcal{K}_{S,\tau'}$ may not have much to say on the properties of 
$\mathcal{K}_{S,\tau}$.

Robinson's infinite forcing is loosely inspired by Cohen's forcing method and gives an elegant formulation 
of the notion of  model companionship: a $\tau$-theory $T$ is the model companion of a 
$\tau$-theory $S$ if it is model complete and 
the models of $T$ are exactly the infinitely generic structures for Robinson's infinite 
forcing applied to models of $S$. 
In \cite{VenturiRobinson}  was described a fundamental connection between the notion of being an infinitely
generic structure and that of being a structure satisfying certain types of forcing axioms. This 
suggests an interesting parallel between a semantic approach
\emph{\`a la Robinson} to the study of the models of set theory and generic absoluteness results.

The main result of this paper (Thm. \ref{thm:mainthm1}) shows that ---in a natural extension of the language of set theory
(given by the addition of predicates for $\Delta_0$ formulas and all universally Baire sets of reals)---
the existence of class many Woodin 
cardinals implies that the model companion of the theory of the universe of all sets is the  theory of $H_{\omega_1}$,
and moreover that for $\Pi_2$-sentences provability overlaps with consistency and 
with forcibility. 
We consider our expansion of the language 
natural, because the  added predicates are exactly those describing  sets of reals 
whose truth properties are unaffected
by the forcing method, and for which, therefore, we have a concrete and stable understanding of their behaviour. 
For example Borel sets of reals are universally Baire, all sets of reals defined by a $\Delta_0$-formula are 
universally Baire, and, assuming large cardinals, all universally Baire sets of reals have all the desirable 
regularity properties such as: Baire property, Lebesgue measurability, perfect set property, determinacy, etc.
Moreover, assuming class many Woodin cardinals, such sets form a point-class closed under: projections, countable 
unions and intersections, complementation, continous images, etc.

We also remark that:

\begin{itemize}

\item On the one hand
Hirschfeld \cite{Hir} showed that any extension of $\ZF$  
has a model companion in the signature $\bp{\in}$.
His result however is uninformative (a consideration he himself made in \cite{Hir}), since the model companion of $\ZF$ for the signature $\bp{\in}$
turns out to be (a small variation of) the theory of dense linear orders; a theory for a binary relation
which has not much in common with our understanding of the $\in$-relation. We consider this fact another indication of the naturalness
of our choice of extending the standard first order language used to formalize set theory.
Indeed, in the standard language containing only the $\in$-relation, there are many set-theroetical concepts
which we consider basic, but whose formalization in first order logic is syntactically quite complex. For example being an ordered pair is a $\Delta_0$-property, but it is only 
$\Sigma_2$-expressible in the signature $\bp{\in}$. This discrepancy causes the ``anomaly'' of Hirschfeld's result, which is here resolved by adding predicates for all the concepts which are sufficiently simple and stable across the different models of set theory; these include the
$\Delta_0$-properties and the universally Baire predicates.

\item
On the other hand (and unlike Hirschfeld's result)
the results of this paper have a highly non-constructive flavour and to be rightly understood require to embrace a 
fully platonistic perspective on the onthology of sets
% to be meaningfully formulated: 
This is why in this paper we assume that the universe of sets $V$ and the
family of hereditarily countable sets $H_{\omega_1}$ are rightful elements of our semantics, 
which ---whenever endowed with suitably defined predicates and constants-- give well-defined first order structures 
for the appropriate signature.
Of course it is possible  to reformulate our results so to make them compatible with a formalist approach to 
set theory \emph{\`a la Hilbert}, but in this case their meaning would be much less transparent;
 hence we refrain here from pursuing this matter further.
 On the other hand this weak point of our results will be completely addressed and resolved in the forthcoming \cite{VIATAMSTI}. 
 Then, in \cite{VIATAMSTII} we will investigate the correspondence between generic absoluteness
 for $H_{\omega_2}$ assuming strong forcing axioms (see the monograph \cite{HSTLARSON} on $(*)$, and 
 the second author's \cite{VIAASP,VIAAUD14,VIAMM+++,VIAMMREV,VIAUSAX})
 and model companionship.
\end{itemize}

The main philosophical thesis we draw from the results of the present paper is that
the success of large cardinals in solving problems of second-order 
arithmetic\footnote{All problems of second order arithmetic are first order properties of $H_{\omega_1}$.} 
via determinacy is due to the fact that these axioms make (in the appropriate language) the theory of 
$H_{\omega_1}$ the model-companion of the theory of $V$, and in particular a model complete theory.

%
%Similar considerations
%can be drawn for other axioms (such as forcing axioms or the constructibility axiom $V=L$)
%which are able to decide most of the problems which cannot be settled on the basis of $\ZFC$ alone.
%In particular we show that 
%if one has a simply definable well-order of $H_{\omega_2}$ (which is the case
%assuming the bounded proper forcing axioms hold),
%then one has simply definable Skolem functions producing witnesses
%of $\Delta_0$-properties. 
%In which case one can easily prove that the first order theory of $H_{\omega_2}$
%is the model companion of the universe of sets in a signature with parameters for all elements of 
%$H_{\omega_2}$, predicates for all bounded formulae, and Skolem functions for such predicates. We can see this result as a companion to the various generic absoluteness results for the theory of $H_{\omega_2}$ assuming forcing axioms the second author has recently presented in \cite{VIAASP,VIAAUD14,VIAMM+++,VIAMMREV,VIAUSAX}.
%We prove as well that $\ZFC+V=L$ is model complete with respect to a natural appropriate first order language.

%\tableofcontents

The paper is structured as follows: 
\begin{itemize}
\item
\S \ref{sec:modeltheoreticcompl} recalls the basic facts on model companionship and on Robinson's infinite forcing. 

\item
In \S \ref{sec:modelcompanZFCgen} we offer reasons for the necessity of an expansion of the language of set theory,
which includes at least predicates for the $\Delta_0$-properties, and eventually also for the universally Baire predicates.

\item
\S \ref{sec:boolvalmod} first recalls few important results on boolean-valued structures and generic absoluteness. 
Then we perform and justify the extension of the first order language of set theory, roughly described before, 
so to include predicates for all $\Delta_0$-formulae; after relativizing the notion of model completeness 
to the generic multiverse, Theorem \ref{thm:omegaVmodcompl} shows that for this expanded language
the theory of $H_{\omega_1}$ is the model companion of the theory of $V$ relative to the generic multiverse. 

\item 
\S \ref{sec:nonmodcompletHomega1} shows that the theory of $H_{\omega_1}$ in the signature with predicates for
the $\Delta_0$-properties is not model complete.

\item
\S \ref{sec:modcompanUBpred} gives the proof of Theorem \ref{thm:mainthm1} 
showing that in a language admitting predicates for all the universally Baire sets,
the theory of $H_{\omega_1}$ is the model companion of the theory of $V$ if we assume the existence of class 
many Woodin cardinals. We also outline why this result 
shows that forcibility, provability, and consistency overlaps for $\Pi_2$-sentences in this expanded signature. 
\end{itemize}

\section{Model theoretic background} \label{sec:modeltheoreticcompl}

We analyze certain classes of first order structures in a given first order signature 
$\tau$ and we will be interested just in theories consisting of sentences.
To fix notation, if $T$ is a first order theory in the signature $\tau$, $\mathcal{M}_T$ denotes the 
$\tau$-structures which are models of $T$.
% and $T_\forall$ denotes the $\tau$-theory given by the $\Pi_1$-sentences for 
% $\tau$ which are provable from $T$.

\begin{definition}
A theory $T$ is \emph{model complete} if for all models $\mathcal{M}$ and $\mathcal{N}$ of $T$ we have that 
$\mathcal{M} \sqsubseteq \mathcal{N}$ ($\mathcal{M}$ is a substructure of $\mathcal{N}$)
implies $\mathcal{M} \prec \mathcal{N}$ ($\mathcal{M}$ is an elementary substructure of $\mathcal{N}$). 
\end{definition}

\begin{definition}
Let $\tau$ be a first order signature and $T$ be a theory for $\tau$.
Given two models $\mathcal{M}$ and $\mathcal{N}$ of a theory $T$
\begin{itemize}
\item
$\mathcal{M}$ is \emph{existentially closed} in $\mathcal{N}$ ($\mathcal{M} \prec_1 \mathcal{N}$)
if the existential and universal formula with parameters in $\mathcal{M}$ have 
the same truth value in $\mathcal{M}$ and $\mathcal{N}$. 
\item
$\mathcal{M}$ is existentially closed for $T$ if it is existentially closed in all its $\tau$-superstructures which are models of $T$.
\end{itemize}
\end{definition}

$\mathcal{E}_T$ denotes the class of $\tau$-models which are existentially closed for $T$. 

Note that in general models in $\mathcal{E}_T$ need not be models\footnote{For example let 
$T$ be the theory of commutative rings with no zero divisors which are not algebraically closed fields. Then 
$\mathcal{E}_T$ is exactly the class of algebraically closed fields and no model in $\mathcal{E}_T$ is a model of $T$.} of $T$. 
Model completeness describes exactly when this is the case.

\begin{lemma}\cite[Lemma 3.2.7]{TENZIE}
(Robinson's test) Let $T$ be a theory. The following are equivalent:
\begin{enumerate}
\item $T$ is model complete. 
\item For every $\mathcal{M} \sqsubseteq \mathcal{N}$ models of $T$
$\mathcal{M} \prec_1 \mathcal{N}$.
\item $\mathcal{E}_T=\mathcal{M}_T$.
\item Each $\tau$-formula is equivalent, modulo $T$, to a universal $\tau$-formula. 
\end{enumerate}
\end{lemma}

Model completeness comes in pair with another fundamental concept which generalizes to arbitrary first order 
theories the relation existing between algebraically closed fields and commutative rings without zero-divisors. As a matter of fact, the case described below
occurs when $T^*$ is the theory of algebraically closed fields
and 
$T$ is the the theory of comutative rings with no zero divisors.

\begin{definition}
Given two theories $T$ and $T^*$, in the same language $\tau$, 
$T^*$ is the \emph{model companion} of $T$ if the following conditions holds:
\begin{enumerate}
\item Each model of $T$ can be extended to a model of $T^*$.
\item Each model of $T^*$ can be extended to a model of $T$.
\item $T^*$ is model complete. 
\end{enumerate}
\end{definition}

The model companion of a theory does not necessarily exist, but, if it does, it is unique.

\begin{theorem}\cite[Thm. 3.2.9]{TENZIE}
A theory $T$ has, up to equivalence, at most one model companion $T^*$.
\end{theorem}

Different theories can have the same model companion: for example the theory of fields 
and the theory of commutative rings with 
no zero-divisors which are not fields both have the theory of algebraically closed fields  
as their model companion.  

\begin{remark}
Using the fact that a theory $T$ is mutually consistent with its model companion $T^*$, 
i.e. the models of one theory can be extended to a model of the other theory
and vice-versa, together with the fact that universal
theories are closed under sub-models it is easy to show that a theory and its model companion agree on 
their universal sentences. 
\end{remark}

\begin{notation}
In what follows, given a theory $T$, $T_{\forall}$ denotes the collection of all 
$\Pi_1$-sentences which are logical consequences of $T$. Similarly
$T_{\exists}$ and $T_{\forall \exists}$ denote, respectively, 
the $\Sigma_1$ and the $\Pi_2$-theorems of $T$. 
\end{notation}

An important properties of $T$-ec models is that they are actually $T_\forall$-ec.
Using this fact one can prove the following:

\begin{theorem}
Let $T$ be a first order theory. If its model companion $T^*$ exists, then
\begin{enumerate}
\item $T_{\forall} = T^*_{\forall}$.
\item $T^*$ is the theory of the existentially closed models of $T_{\forall}$. 
\item $T^*$ is axiomatized by $T_{\forall \exists}$. 
\end{enumerate}
\end{theorem}

Possibly inspired by Cohen's forcing method,
Robinson  introduced  what is now called Robinson's infinite forcing \cite{HIRWHE75}.
In this paper we are interested in a slight generalization of Robinson's definition which makes the class of models over which
we define infinite forcing an additional parameter.

\begin{definition}
Given a class of structure $\mathcal{C}$ for a signature $\tau$, \emph{infinite forcing for $\mathcal{C}$} is recursively defined as follows
for a $\tau$-formula $\phi(x_1,\dots,x_n)$,  a structure $\mathcal{M} \in \mathcal{C}$ with domain $M$ and $a_1,\dots,a_n\in M$: 
\begin{itemize}
\item
For $\phi(x_1,\dots,x_n)$ atomic, 
$\mathcal{M} \VDash_\mathcal{C}\varphi(a_1,\dots,a_n)$
if and only if $\mathcal{M} \models\varphi(a_1,\dots,a_n)$; 
\item
$\mathcal{M}  \VDash_\mathcal{C}  \varphi(a_1,\dots,a_n)\land \psi(a_1,\dots,a_n)$ if and only if 
$\mathcal{M} \VDash_\mathcal{C} \varphi(a_1,\dots,a_n)$ and 
$\mathcal{M}  \VDash_\mathcal{C} \psi(a_1,\dots,a_n)$;  
\item
$\mathcal{M}  \VDash_\mathcal{C}  \varphi(a_1,\dots,a_n) \lor \psi(a_1,\dots,a_n)$ if and only if 
$\mathcal{M}  \VDash_\mathcal{C} \varphi(a_1,\dots,a_n)$ or 
$\mathcal{M} \VDash_\mathcal{C} \psi(a_1,\dots,a_n)$;  
\item
$\mathcal{M}  \VDash_\mathcal{C} \forall x\varphi(x,a_1,\dots,a_n)$  if and only if 
(expanding $\tau$ with constant symbols for all elements of $M$)
$\mathcal{M} \VDash_\mathcal{C} \varphi(a,a_1,\dots,a_n)$, for every $a \in M$; 
\item
$\mathcal{M}  \VDash_\mathcal{C} \neg \varphi(a_1,\dots,a_n)$  if and only if 
$\mathcal{N} \not\VDash_\mathcal{C}\varphi(a_1,\dots,a_n)$ for all 
$\mathcal{N}\in \mathcal{C}$ superstructures of $\mathcal{M}$. 
\end{itemize}
\end{definition}
Robinson's infinite forcing consider only the case in which $\mathcal{C}=\mathcal{M}_T$. We are interested in considering Robinson's
infinite forcing also in case
$\mathcal{C}$ is not of this type.

As in the case of Cohen's forcing, this method produces objects that are generic. In this case generic models. 

\begin{notation}
Given a class of structure $\mathcal{C}$ for a signature $\tau$
A structure $\mathcal{M}\in\mathcal{C}$ is \emph{infinitely generic for $\mathcal{C}$} 
whenever satisfaction and infinite forceability coincide: i.e., for every formula $\varphi(x_1,\dots,x_n)$
and $a_1,\dots,a_n\in M$, 
we have 
\[
\mathcal{M} \vDash \varphi(a_1,\dots,a_n) \iff \mathcal{M} \VDash_\mathcal{C} \varphi(a_1,\dots,a_n).
\]
By  $\mathcal{F}_\mathcal{C}$, we indicate the class of infinitely generic structures for $\VDash_{\mathcal{C}}$.
\end{notation}

Generic structures capture semantically the syntactic notion of model companionship. 

\begin{theorem}
Let $T$ be a theory in a signature $\tau$.
The following  are equivalent:
\begin{enumerate}
\item $T^*$ exists.
\item $\mathcal{E}_T$ is an elementary class.
\item $\mathcal{F}_T$ is an elementary class.
\item $\mathcal{E}_T=\mathcal{F}_{\mathcal{M}_{T_\forall}}$ (i.e. the existentially closed structures for $T$ are the generic structures for Robinson's infinite forcing applied to the class $\mathcal{M}_{T_\forall}$).
\end{enumerate}
\end{theorem}

Motivated by the above theorem we will analyze the generic multiverse with in mind Robinson's 
characterization of model companionship by means of infinite forcing.

\section{Model companions for $\ZFC$: do they exist?} \label{sec:modelcompanZFCgen}

We already outlined that the model completeness of a theory 
is sensitive to the language in which that theory is expressed. 
We now embark in the task of selecting the right 
first order language to use for the construction of the model companion of 
(extensions of) $\ZFC$. We will first argue that (at least for the purposes of studying set theory by means of first order logic) 
this is neither the language $\{\in\}$ nor the language $\{\in,\subseteq\}$, even if these are the languages in which set theory 
is usually formalized in almost all textbooks.

As a preliminary result, we have that the model companion of $\ZF$ for the language 
$\{\in\}$ 
has already been fully described. 
\begin{theorem}
(Hirschfeld \cite[Thm. 1, Thm. 5]{Hir}) The universal theory of any $T\supseteq \ZF$ in the signature $\bp{\in}$
is the theory
\[
S = \bp{\forall x_1\dots\forall x_n(x_1\notin x_2\vee x_2\notin x_3\vee\dots\vee 
x_{n-1}\notin x_n\vee x_n\notin x_1): \,n\in\mathbb{N}}.
\]
Letting for $A\subseteq n$ 
\[
\delta_A(x_1,\dots,x_n,y)=\bigwedge_{i\in A} x_i\in y\wedge\bigwedge_{i\not\in A} x_i\not\in y,
\]
the model companion of $\ZF$ is the theory
\begin{align*}
S^*=&\bp{\forall x_1\dots x_n \exists y\,\delta_A(x_1,\dots,x_n,y): n\in\omega, \,A\subseteq n}\cup\\
&\cup\bp{\forall x,y\,\exists z[x=y\vee (x\in z\wedge z\in y)\vee(y\in z\wedge z\in x)]}.
\end{align*}
\end{theorem}

In particular $S^*$ is also the model companion of $\ZFC$, given that $S$ is the universal theory of any $T\supseteq\ZF$, among which $\ZFC$.

Notice that $S$ only says that the graph of the $\in$-relation has no loops, while Hirschefeld also shows that
in every model of $S^*$ the formula $\exists x(a\in x\wedge x\in b)$ 
defines a dense linear order without endpoints \cite[Thm. 3]{Hir}.
In particular there is no apparent relation between the meaning of the $\in$-relation in a model of $\ZF$ (in its standard
models it is a well-founded relation not linearly ordered) and the meaning of the $\in$-relation in models of 
$S^*$ (it is a dense linear order without end-points).

We believe (as Hirschfeld) that the above result gives a clear mathematical insight of why the language $\bp{\in}$ is not expressive 
enough to describe the ``right'' model companion of set theory.
A key issue is the following: we are inclined to consider concepts and properties which can be formalized
by formulae with bounded quantifiers much simpler and concrete 
than those which can only be 
formalized by formulae which make use of unrestricted quantification. 
This is reflected by the fact that properties formalizable by means of formulae 
with bounded quantifiers are absolute between transitive models of $\ZFC$. This fact fails badly
for properties defined by means of unbounded quantification.

For example 
the property \emph{$f$ is a function} is expressible using only bounded quantification, while the property
\emph{$\kappa$ is a cardinal} is not. It is well known that the former is a property that is absolute between transitive models of 
$\ZFC$ containing $f$, while the latter is not. However if one tries to formalize by means of a first order formula in the signature $\in$
the formula \emph{$f$ is a function}, one realizes that any such formalization is a very complicated formula.
For example already the formalization of the concept \emph{$x$ is the ordered pair with first component $y$ and second component $z$}  by means of Kuratowski  gives an $\in$-formula $\phi(x,y,z)$ which is $\Sigma_2$:
\[
\exists t\exists u\;[\forall w\,(w\in x\leftrightarrow w=t\vee w=u)
\wedge\forall v\,(v\in t\leftrightarrow v=y)
\wedge\forall v\,(v\in u\leftrightarrow v=y\vee v=z)].
\]
 It is also a matter of fact that absolute properties are regarded as 
``tame'' set theoretic properties (as their truth value cannot be changed by forcing, e.g \emph{$f$ is a function} 
 remains true in any transitive model to which $f$ belongs), 
while non-absolute ones are more difficult to control (they are not immune to forcing). For example, when $\kappa$ is an uncountable 
cardinal of the ground model, it will  cease to be so in any generic extension by 
$\Coll(\omega,\kappa)$.

These are good reasons that lead to formalize set theory in a first order language able to recognize syntactically the different semantic 
complexity of absolute and non-absolute concepts. As Hirschfeld showed, this is not the case for the 
$\ZF$-axioms in the language $\bp{\in}$.

In Kunen and Jech's standard textbooks the solution adopted is that of passing from first order logic to a 
logic with bounded quantifiers $\exists x\in y$ and $\forall x\in y$, binding the variable $x$, so that
$\exists x\in y\phi(x,y,\vec{z})$ is logically equivalent to $\exists x(x\in y\wedge\phi(x,y,\vec{z}))$ and 
$\forall x\in y\phi(x,y,\vec{z})$ is logically equivalent to $\forall x(x\in y\rightarrow\phi(x,y,\vec{z}))$.
In this new logic \emph{$f$ is a function} is expressible by a formula with only bounded quantifiers, while
\emph{$\kappa$ is a cardinal} is expressible by a formula of type $\forall x\phi(x,\kappa)$ with $\phi$ having only bounded
quantifiers. 
However, Kunen and Jech's solution is not convenient for the scopes of this paper, because it 
formalizes set theory outside first order logic, making less transparent how we could use 
model theoretic techniques (designed expressly for first order logic) 
to isolate what is the correct model companion of set theory.
The alternative solution we adopt in this paper is that of expressing set theory in a first order 
language with relational symbols for any bounded formula.

\begin{notation}\label{not:keynotation}
\emph{}

\begin{itemize}
\item
$\tau_{\mathsf{ST}}$ is the extension of the first order signature $\bp{\in}$ for set theory 
which is obtained 
by adjoining predicate symbols
$R_\phi$ of arity $n$ for any $\Delta_0$-formula $\phi(x_1,\dots,x_n)$, 
and constant symbols for 
$\omega$ and $\emptyset$.
%\item
%$\tau_{\kappa}$ is the extension of $\tau_{\ST}$
%obtained by adding a constant symbol for $\kappa$.

\item $\ZFC^{-}$ is the
$\in$-theory given by the axioms of
$\ZFC$ minus the power-set axiom.

\item
$T_{\ST}$ is the $\tau_{\ST}$-theory given by the axioms
\[
\forall \vec{x} \,(R_{\forall x\in y\phi}(y,\vec{x})\leftrightarrow \forall x(x\in y\rightarrow R_\phi(y,x,\vec{x}))
\]
\[
\forall \vec{x} \,[R_{\phi\wedge\psi}(\vec{x})\leftrightarrow (R_{\phi}(\vec{x})\wedge R_{\psi}(\vec{x}))]
\]
\[
\forall \vec{x} \,[R_{\neg\phi}(\vec{x})\leftrightarrow \neg R_{\phi}(\vec{x})]
\]
%\[
%(\forall \vec{x}\exists!y \,R_{\phi}(y,\vec{x}))\rightarrow (\forall \vec{x}\,R_{\phi}(f_\phi(\vec{x}),\vec{x}))
%\]
for all $\Delta_0$-formulae $\phi(\vec{x})$, together with the $\Delta_0$-sentences
\[
\forall x\in\emptyset\,\neg(x=x),
\]
\[
\omega\text{ is the first infinite ordinal}
\]
(the former is an atomic $\tau_{\ST}$-sentence, the latter is expressible as the atomic sentence for 
$\tau_{\ST}$ stating that
$\omega$ is a non-empty limit ordinal and all its elements are successor ordinals or $0$).
\item
$\ZFC^-_{\ST}$ is the $\tau_{\ST}$-theory 
\[
\ZFC^{-}\cup T_{\ST},
\] 
accordingly we define $\ZFC_{\ST}$.
\end{itemize}
\end{notation}

%
%
%\begin{definition}
%Given the first order signature
%\[
%\tau_{\ST}=\bp{R_\phi: \phi\text{ logically equivalent to a bounded formula in the signature $\bp{\in}$}},
%\] 
%$\ZFC_{\ST}$ is the $\tau_{\ST}$-theory obtained
%adding to $\ZFC$ 
%the axioms
%\[
%\forall\vec{x}(\phi(\vec{x})\leftrightarrow R_\phi(\vec{x}))
%\]
%for all formulae $\phi(\vec{x})$ logically equivalent to a bounded formula.
%\end{definition}

In $\ZFC_{\ST}$ we now obtain that many absolute concepts (such as that of being a function) are now expressed
by atomic formulas, while other more complicated ones 
(like for example those defined by means of transfinite recursion over an absolute property, such as 
\emph{$x$ is the transitive closure of $y$}) can still be expressed by means
of $\ZFC^{-}_{\ST}$-provably $\Delta_1$-properties of $\tau_{\ST}$ (i.e. properties which are $\ZFC^{-}_{\ST}$-provably equivalent
 at the same time to a $\Pi_1$-formula
and to a $\Sigma_1$-formula), which therefore are still absolute between any two models (even non-transitive)
$\mathcal{M}$, $\mathcal{N}$ of $\ZFC_{\ST}$  of which one is a substructure of the other. 
On the other hand many definable properties have truth values which may vary depending on which model of 
$\ZFC_{\ST}$ we work in  (for example \emph{$\kappa$ is an uncountable cardinal} is a $\Pi_1\setminus\Sigma_1$-property in 
$\ZFC_{\ST}$ whose truth value may depend on the choice of the model of $\ZFC_{\ST}$ to which $\kappa$ belongs).

Our first aim is to identify what could be the model companion of $\ZFC_{\ST}$.
To this aim, first recall that Levy's absoluteness gives that 
$H_{\omega_1}\prec_{\Sigma_1}V$, and that for any set $X$ there is a forcing extension in which 
$X$ is countable (just force with $\Coll(\omega,X)$). In particular one can argue that 
the $\Pi_2$-assertion $\forall X\exists f:\omega\to X\emph{ surjectve}$ is generically true for Robinson's infinite forcing applied to the forcing extensions of $V$. Notice that 
$H_{\omega_1}\models\forall X\exists f:\omega\to X\emph{ surjectve}$.

The natural conjecture is to infer that the $\tau_{\ST}$-theory of $H_{\omega_1}$ 
is the model companion of the $\tau_{\ST}$-theory of $V$. 
We now show exactly to which extent the conjecture is true, while proving that it is false.

Towards this aim we conclude this section introducing a semantic and relativized 
notion of model completeness and model companionship.

%
%\begin{definition}
%Given a theory $T$ and a category $(\mathcal{M},\to_{\mathbb{M}})$ with $\mathbb{M}$ a class 
%of models of $T$ and $\to_{\mathbb{M}}$ a class of morphisms between them, 
%$T$ is \emph{model complete with respect to $\mathbb{M}$} if for all models $\mathcal{M}$ and 
%$\mathcal{N}$ in $\mathbb{M}$ we have that 
%$\mathcal{M} \prec_1 \mathcal{N}$, whenever there is a morphism $f: \mathcal{M} \to \mathcal{N}$ in $\to_{\mathbb{M}}$. 
%\end{definition}

\begin{definition}
Let $\tau$ be a first order signature.
Given a category $(\mathcal{C},\to_{\mathcal{C}})$ of $\tau$-structures (elements of $\mathcal{C}$) and morphisms between them (elements of $\to_\mathcal{C}$),
a class $\mathcal{D}\subseteq\mathcal{C}$ is the class of its generic structures
if:
\begin{itemize}
\item For every structure $\mathcal{M}$ in $\mathcal{C}$ there is $\mathcal{N}\in\mathcal{D}$ and 
$i:\mathcal{M}\to\mathcal{N}$ in $\to_\mathcal{C}$. 
\item For every $i:\mathcal{M}\to\mathcal{N}$ in $\to_{\mathcal{C}}$ with $\mathcal{M},\mathcal{N}$  in $\mathcal{D}$, 
$i$ is $\Sigma_1$-elementary, i.e.  $i[\mathcal{M}]\prec_1\mathcal{N}$.
\end{itemize} 
If $\mathcal{D}=\mathcal{C}$ we say that the category $(\mathcal{C},\to_{\mathcal{C}})$ is model complete.
\end{definition}
In particular if $\mathcal{C}$ is the class of models of $T$, $\to_\mathcal{C}$ is the class of all morphisms between
models of $T$, and $\mathcal{D}=\mathcal{M}_S$, $S$ is the model companion of $T$.

\section{Boolean valued models and generic absoluteness}\label{sec:boolvalmod}

Our first aim is to outline which first order properties are first order invariant with respect to the forcing method.
Toward this aim we recall some standard facts on boolean-valued models for set theory, giving appropriate references 
for the relevant proofs (in particular \cite{BELLSTBVM}, or \cite{VIAAUDSTE13}, the forthcoming \cite{VIAAUDSTEBOOK},
the notes \cite{viale-notesonforcing}). 
In what follows we do not consider languages with function symbols in order to avoid some technical difficulties.

\begin{definition}
Let $\LL=\bp{R_i: i\in I}\cup\bp{c_j. j\in J}$ be a relational language, and $\bool{B}$ a  boolean algebra. 
A $\bool{B}$-valued model for $\LL$ is a tuple 
$\M=\bp{M}\cup\bp{=^\M_{\bool{B}}}\cup\bp{R_{i\bool{B}}^\M: i\in I}\cup\bp{c_j^\M: j\in J}$, where:
\begin{enumerate}
\item $M$ is a non-empty set;
\item $=^\M_{\bool{B}}$ is the boolean value of the equality symbol, i.e. a function
\begin{align*}
=^\M_{\bool{B}}:&\ M^2\to\bool{B};\\
\ap{x,&y}\mapsto\Qp{x=y}^\M_{\bool{B}}
\end{align*}
\item $R_{\bool{B}}^\M$ is the interpretation of the relational symbol $R$. If $R$ has arity $n$,
\begin{align*}
R_{\bool{B}}^\M:&\ M^n\to\bool{B};\\
\ap{x_1, \dots,&x_n}\mapsto\Qp{R(x_1, \dots, x_)}^\M_{\bool{B}}
\end{align*}
\item $c_j^\M\in M$ is the interpretation of the constant symbol $c_j$.
\end{enumerate}
\noindent
We require that the following conditions hold:
\begin{itemize}
\item for all $x, y, z\in M$,
\begin{equation}
\Qp{x=x}^\M_{\bool{B}}=1_{\bool{B}},
\end{equation}
\begin{equation}
\Qp{x=y}^\M_{\bool{B}}=\Qp{y=x}^\M_{\bool{B}},
\end{equation}
\begin{equation}
\Qp{x=y}^\M_{\bool{B}}\wedge\Qp{y=z}^\M_{\bool{B}}\leq\Qp{x=z}^\M_{\bool{B}};
\end{equation}
\item if $R\in \LL$ is a $n$-ary relational symbol, for every $\ap{x_1, \dots, x_n}, \ap{y_1, \dots, y_n}\in M^n$,
\begin{equation}
\Bigl(\bigwedge_{i=1}^n\Qp{x_i=y_i}^\M_{\bool{B}}\Bigr)\wedge
\Qp{R(x_1, \dots, x_n)}^\M_{\bool{B}}\leq\Qp{R(y_1, \dots, y_n)}^\M_{\bool{B}}.
\end{equation}
\end{itemize}

\end{definition}

From here on, if no confusion can arise, we avoid to put the superscript $\M$ and the subscript $\bool{B}$. Moreover, we will write $\M$ or $M$ equivalently to indicate a boolean valued model or its underlying set.

In general it makes sense to define the first order semantics also for certain $\bool{B}$-valued models with $\bool{B}$ non complete.
However in this paper we can limit to define this semantics only for $\bool{B}$-valued models with $\bool{B}$ a complete boolean algebra.
%Let us now assume the boolean algebra $\bool{B}$ to be complete.
\begin{definition}
Let $\LL=\bp{R_i: i\in I}\cup\bp{c_j. j\in J}$ be a relational language, $\bool{B}$ a complete boolean algebra,
$\mathcal{M}$ a $\bool{B}$-valued model. 
We evaluate the formulae of $\LL(M):=\LL\cup\bp{c_\tau: \tau\in M}$ without free variables in the following way:
\begin{itemize}
\item $\Qp{R(c_{\tau_1}, \dots, c_{\tau_n})}:=\Qp{R(\tau_1, \dots, \tau_n)}$;
\item $\Qp{\varphi\wedge\psi}:=\Qp{\varphi}\wedge\Qp{\psi}$;
\item $\Qp{\neg\varphi}:=\neg\Qp{\varphi}$;
\item $\Qp{\varphi\rightarrow\psi}:=\neg\Qp{\varphi}\vee\Qp{\psi}$;
\item $\Qp{\exists x\varphi(x, c_{\tau_1}, \dots, c_{\tau_n})}:=
\bigvee_{\sigma\in M}\Qp{\varphi(c_\sigma, c_{\tau_1}, \dots, c_{\tau_n})}$;
\item $\Qp{\forall x\varphi(x, c_{\tau_1}, \dots, c_{\tau_n})}:=
\bigwedge_{\sigma\in M}\Qp{\varphi(c_\sigma, c_{\tau_1}, \dots, c_{\tau_n})}$.
\end{itemize}
Given an assignment $\mathcal{\nu}$ of the free variables to $\mathcal{M}$ and an $\LL$-formula $\phi(x_1,\dots,x_n)$
\[
\Qp{\varphi(x_1, \dots, x_n)}^{\mathcal{M},\nu}_\bool{B}=\Qp{\varphi(\nu(x_1), \dots, \nu(x_n))}.
\]
\end{definition}
We write $\Qp{\varphi(\tau_1, \dots,\tau_n)}$ rather than $\Qp{\varphi(c_{\tau_1}, \dots, c_{\tau_n})}^{\mathcal{M}}_\bool{B}$.

Observe that, if $\bool{B}=\bp{0, 1}$, a $\bool{B}$-model is simply a Tarski structure for the language $L$, 
and the semantic we have just defined is the Tarski semantic.
\begin{definition}
A statement $\varphi$ in the language $L$ is \emph{valid} in a $\bool{B}$-valued model $\M$ for $L$ if 
$\Qp{\varphi}^\M_{\bool{B}}=1_{\bool{B}}$. A theory $T$ is valid in $\M$ if every axiom of $T$ is valid in $\M$.
\end{definition}
\noindent
It can be proved (see the proof of \cite[Theorem 4.1.5]{viale-notesonforcing}) that, if $\varphi(x_1, \dots, x_n)$ is a formula 
with free variables $x_1, \dots, x_n$ and $\tau_1, \dots, \tau_n, \sigma_1, \dots, \sigma_n\in M$, then
\begin{equation}
\Qp{\tau_1=\sigma_1}\wedge\dots\wedge\Qp{\tau_n=\sigma_n}\wedge\Qp{\varphi(\tau_1, \dots, \tau_n)}\leq
\Qp{\varphi(\sigma_1, \dots, \sigma_n)}.
\end{equation}
\noindent
From here on, we will consider this fact as granted.
\begin{definition}
Let $\bool{B}$ be a complete boolean algebra and let $L=\bp{R_i: i\in I,\, c_j:j\in J}$ be a first order relational language. 
Let $\M=\bp{M, R_i^\M: i\in I}$ be a $\bool{B}$-valued model. Let $F$ be a filter in $\bool{B}$. 
We define the \emph{quotient} $\M/_F=\bp{M/_F, R_i/_F}$ of $\M$ by $F$ as follows:
\begin{itemize}
\item $M/_F:=\bp{[\tau]_F: \tau\in M}$, where $[\tau]_F:=\bp{\tau\in M: \Qp{\tau=\sigma}\in F}$;
\item $\Qp{R_i([\tau_1]_F, \dots, [\tau_n]_F)}^{\M/_F}:=\Bigl[\Qp{R_i(\tau_1, \dots, \tau_n)}^\M\Bigr]_F\in\bool{B}/_F$ 
for every $i\in I$.
\item $c_j$ is interpreted by $[c_j^{\mathcal{M}}]_F$.
\end{itemize}
\end{definition}
\noindent
It is possible to see that $\M/_F$ is a $\bool{B}/_F$-valued model (even if $\bool{B}/_F$ is not complete). 
In particular, if $U$ is a ultrafilter, $\M/_U$ is a $2$-valued model, i.e. a classical Tarski structure.

\begin{definition}
Assume $\bool{B}$ is a complete boolean algebra.
A $\bool{B}$-valued model $\mathcal{M}$ for the signature $\mathcal{L}$ is \emph{full} if 
for any $\mathcal{L}$-formula $\phi(x_0,\dots,x_n)$ and  $\tau_1,\dots,\tau_n\in \mathcal{M}$
\[
\Qp{\exists x\phi(x,\tau_1,\dots,\tau_n)}^{\mathcal{M}}_{\bool{B}}=\Qp{\phi(\sigma,\tau_1,\dots,\tau_n)}^{\mathcal{M}}_{\bool{B}}
\]
for some $\sigma\in \mathcal{M} $.
\end{definition}

\begin{lemma}[Lemma 14.14 \cite{JECH}]
Assume $\bool{B}$ is a complete boolean algebra and
$\mathcal{M}$ is a full $\bool{B}$-valued model for $\LL$.
Then for all $\tau_1,\dots,\tau_n\in\mathcal{M}$ and ultrafilter $G$ on $\bool{B}$
\[
\mathcal{M}/_G\models \phi([\tau_1]_G,\dots,[\tau_n]_G) \text{ if and only if } \Qp{\phi(\tau_1,\dots,\tau_n)}^\mathcal{M}\in G.
\]
\end{lemma}

A property implying fullness is the mixing property:

\begin{definition}
Let $\bool{B}$ be a complete boolean algebra and
$\mathcal{M}$ a $\bool{B}$-valued model.

$\mathcal{M}$ has the mixing property if
for all $A$ antichains on $\bool{B}$ and $\bp{\tau_A:a\in A}\subseteq \mathcal{M}$ there exists 
$\tau \in \mathcal{M}$ such that
\[
\Qp{\tau=\tau_a}^{\mathcal{M}}\geq a
\]
for all $a\in A$.
\end{definition}

Note that the mixing property for the $\bool{B}$-valued model $\mathcal{M}$
depends only on the interpretation of the equality symbol, while
fullness depends on the intepretation of all relation symbols of $\LL$ in $\mathcal{M}$.

\begin{lemma}
Let $\bool{B}$ be a complete boolean algebra
Assume $\mathcal{M}$ is a $\bool{B}$-valued model with the mixing property.
Then $\mathcal{M}$ is full.
\end{lemma}
\begin{proof}
This is a variation of \cite[Lemma 14.19]{JECH}. Else see \cite[Proposition 2.1.7]{PIEROBONTHESIS}.
\end{proof}

%\subsection{The models $V^{\bool{B}}$ and $H_{\dot{\kappa}}^{\bool{B}}$}

\subsection{The model companion of set theory for the generic multiverse} \label{sec:modelcompsetth}

Recall that $V$ denotes the universe of all sets, and for any complete boolean algebra 
$\bool{B}\in V$ 
\[
V^{\bool{B}}=\bp{\tau: \, \tau:X\to \bool{B} \text{ is a function with $X\subseteq V^{\bool{B}}$ a set}}
\]
is the boolean valued model for set theory generated by forcing with $\bool{B}$. 

$V^{\bool{B}}$ is endowed with the structure of a $\bool{B}$-valued model for the language of set theory 
$\mathcal{L}=\bp{\in,\subseteq}$, letting (see \cite[Def. 5.1.1]{viale-notesonforcing} for details)
\begin{equation}
\Qp{\tau_1\in\tau_2}_{\bool{B}}=\bigvee_{\sigma\in\dom(\tau_2)}
(\Qp{\tau_1=\sigma}_{\bool{B}}\wedge\tau_2(\sigma)), 
\end{equation}
\begin{equation}
\Qp{\tau_1\subseteq\tau_2}_{\bool{B}}=\bigwedge_{\sigma\in\dom(\tau_1)}
(\neg\tau_1(\sigma)\vee\Qp{\sigma\in\tau_2}_{\bool{B}}), 
\end{equation}
\begin{equation}
\Qp{\tau_1=\tau_2}_{\bool{B}}=
\Qp{\tau_1\subseteq\tau_2}_{\bool{B}}\wedge
\Qp{\tau_2\subseteq\tau_1}_{\bool{B}}.
\end{equation}

By \cite[Thm 1.17]{BELLSTBVM} $V^{\bool{B}}$ is a $\bool{B}$-valued model for $\bp{\in,\subseteq,=}$.
By \cite[Thm 1.33]{BELLSTBVM} all axioms of $\ZFC$ gets boolean value $1_\bool{B}$ in $V^{\bool{B}}$.
By \cite[Lemma 14.18]{JECH} $V^{\bool{B}}$ has the mixing property.

%The boolean value $\Qp{\phi(\tau_1,\dots,\tau_n}_{\bool{B}}$ of formulae
%$\phi(x_1,\dots,x_n)$ with assignment $\tau_1,\dots,\tau_n$ are given according to the standard rules of boolean valued semantics (see for example \cite[Section 4.1]{viale-notesonforcing}); concretely: atomic formulae of type
%$\tau_1\mathrel{R}\tau_2$ are given
%the boolean value $\Qp{\tau_1\mathrel{R}\tau_2}_{\bool{B}}$; the boolean operations allows to define the boolean value associated to a conjunction/disjunction/negation of formulae; completeness of $\bool{B}$ allows to define
%\[
%\Qp{\exists x\phi(x,\vec{\tau})}_{\bool{B}}=\bigvee_{\sigma\in V^{\bool{B}}}\Qp{\phi(\sigma,\vec{\tau})}_{\bool{B}}. 
%\]

The class of models we will analyze is given by the generic extensions of initial segments of $V$. 
To make this precise we
need a couple of definitions.

\begin{definition}\label{def:HkappaB}
Let $\bool{B}$ be a complete boolean algebra.
and $\dot{\kappa}\in V^{\bool{B}}$ be
 such that 
$\Qp{\dot{\kappa}\text{ is a regular cardinal}}_{\bool{B}} =1_{\bool{B}}$.
Given $\kappa\geq\bool{B}$ the least regular cardinal in $V$ such that 
$\Qp{\dot{\kappa}\leq\check{\kappa}}=1_{\bool{B}}$ and $\bool{B}$ is $<\kappa$-CC, let
\[
H_{\dot{\kappa}}^\bool{B}=\bp{\tau\in V^{\bool{B}}\cap H_\kappa^V: 
\Qp{ \tau\text{ has transitive closure of size less than }\dot{\kappa}}_\bool{B}=1_{\bool{B}}}
\]
\end{definition}
We let $\dot{\omega_1}^{\bool{B}}$ and $H_{\dot{\omega_1}}^\bool{B}$ be canonical $\bool{B}$-name for the first uncountable cardinal
and for the family of hereditarily countable sets, i.e. 
\[
\Qp{\dot{\omega_1}^{\bool{B}}\text{ is the first uncountable cardinal}}_{\bool{B}} =1_{\bool{B}},
\]
\[
H_{\dot{\omega_1}}^\bool{B}=H_{\dot{\omega_1}^{\bool{B}}}^\bool{B}.
\]

It is left to the reader to check that for any $\Delta_0$-formula $\phi(\tau_1,\dots,\tau_n)$ for the signature $\in$
the truth value of $\Qp{\phi(\tau_1,\dots,\tau_n)}_\bool{B}$ is the same in $H_{\dot{\kappa}}^\bool{B}$ and in
$V^{\bool{B}}$: by inspecting the definitions one realizes that
the truth value of these formulae is defined by transfinite recursion on a set of names contained in 
$H_{\dot{\kappa}}^\bool{B}$, hence the computations of these truth value is the same in $V^{\bool{B}}$ and in
$H_{\dot{\kappa}}^\bool{B}$. Key to this result is the equality
\[
\Qp{\forall x\in \sigma\phi(x,\sigma,\tau_1,\dots,\tau_n)}_{\bool{B}}=
\bigwedge_{\tau\in\dom{\sigma}}\tau(\sigma)\wedge\Qp{\phi(\sigma,\tau,\tau_1,\dots,\tau_n)}_{\bool{B}},
\]
(see \cite[Exercise 14.12]{JECH}). See also the proof of Lemma \ref{lem:Cohengen} below.

%Key properties of $V^{\bool{B}}$
%and of the models $H_{\dot{\kappa}}^\bool{B}$ defined above
%are the mixing property and fullness.

\begin{lemma}[Mixing Lemma for $H_{\dot{\kappa}}^\bool{B}$]
Assume $\bool{B}$ is a cba and $\dot{\kappa}$ is as in Def. \ref{def:HkappaB}.
Then $H_{\dot{\kappa}}^\bool{B}$ satisfies the mixing property.
\end{lemma}

\begin{proof}
The same proof that works for $V$ (e.g. \cite[Lemma 14.18]{JECH}) works for
$H_{\dot{\kappa}}^\bool{B}$ as well, since the required $\bool{B}$-name $\tau$  construed in that proof
is in $H_{\dot{\kappa}}^\bool{B}$ if $\bp{\tau_A:a\in A}\subseteq H_{\dot{\kappa}}^\bool{B}$. 
\end{proof}

%Now 
%\[
%\Qp{\phi(\tau_1,\dots,\tau_n)}^{V^\bool{B}}_\bool{B}=
%\Qp{\phi(\tau_1,\dots,\tau_n)}^{H_{\dot{\kappa}}^\bool{B}}_\bool{B}
%\]
%holds true for any $\Delta_0$-formula $\phi(x_1,\dots,x_n)$ and any $\tau_1,\dots,\tau_n\in H_{\dot{\kappa}}^\bool{B}$:
%again because in evalutaing the $\bool{B}$-value of a bounded formula $\phi(\tau_1,\dots,\tau_n)$ one must only control the
% truth value of formulae $\psi(\sigma_1,\dots,\sigma_k)$ for $\bool{B}$-names $\sigma_1,\dots,\sigma_k$ 
% occuring in the transitive closure
% of the set $\bp{\tau_1,\dots,\tau_n}$, hence in $H_{\dot{\kappa}}^\bool{B}$.

 The forcing theorem states that:
\begin{itemize}
\item \cite[Thm 4.3.2, Thm 5.1.34]{viale-notesonforcing} (\L o\'s theorem for full boolean valued models)
For all ultrafilter $G$ on $\bool{B}$, $\tau_1,\dots,\tau_n\in V^{\bool{B}}$, and $\tau_{\ST}$-formula $\phi(x_1,\dots,x_n)$
\[
(V^{\bool{B}}/G,\in/_G)\models\phi([\tau_1]_G,\dots,[\tau_n]_G)\text{ if and only if } \Qp{\phi(\tau_1,\dots,\tau_n)}_{\bool{B}}\in G.
\]

\item The same conclusion holds with $H_{\dot{\kappa}}^\bool{B}$ in the place of $V^{\bool{B}}$.

\item \cite[Thm. 5.2.3]{viale-notesonforcing} Whenever $G$ is $V$-generic for $\bool{B}$ 
the map 
\[
[\tau]_G\mapsto \tau_G=\bp{\sigma_G: \exists b\in G\,\ap{\sigma,b}\in\tau}
\]
is the Mostowski collapse of the
class $V^{\bool{B}}/G$ defined in $V[G]$ onto $V[G]$ and its restriction to $H_{\dot{\kappa}}^\bool{B}/_G$ maps the latter onto
$H_{\dot{\kappa}_G}^{V[G]}$.
\end{itemize}
 
%The interpretation of all other $\tau_{\ST}$-formulae follow the same rules given for 
%$V^\bool{B}$, except that in evaluating quantifiers now we let 
%$\sigma$ range just over the appropriate domain $H_{\dot{\kappa}}^\bool{B}$. 

Using the forcing theorem one gets that 
$\Qp{\psi}^{H_{\dot{\kappa}}^\bool{B}}=1_{\bool{B}}$
for any 
axiom $\psi$ of $\ZFC^-_{\ST}$, since
it is the case that for all $G$ $V$-generic for 
$\bool{B}$
\[
H_{\dot{\kappa}}^\bool{B}[G]=\bp{\tau_G: \tau\in H_{\dot{\kappa}}^\bool{B}}=H_{\dot{\kappa}_G}^{V[G]},
\]
i.e. $H_{\dot{\kappa}}^\bool{B}$ is a canonical family of $\bool{B}$-names to denote the
$H_{\dot{\kappa}_G}^{V[G]}$ of the generic extension, and the latter always models $\ZFC^-_{\ST}$.

%The reader not satisfied by the very sketchy arguments we gave for all these nice properties fo the $\bool{B}$-valued model
%$H_{\dot{\kappa}}^\bool{B}$ will find below some more details on why these holds.

%Another key observation is that for all $\Delta_0$-formulae $\phi(x_1,\dots,x_n)$, and all $G$
%The reader not satisfied with this sketchy explanation of why 

%More interestingly the following holds:
%\begin{fact}
%Assume $\phi(x_1,\dots,x_n)$ is a $\Delta_0$-formula. Then
%
%\end{fact}

%
%
%For any ultrafilter $G$ on $\bool{B}$ and $\mathcal{M}$ any structure among $V^{\bool{B}}$ or $H_{\dot{\kappa}}^\bool{B}$,
%$\mathcal{M}/_G$ stands for the class (or set) $\bp{[\tau]_G:\tau\in \mathcal{M}}$,
%where $[\tau]_G=\bp{\sigma\in V^{\bool{B}}: \Qp{\sigma=\tau}_\bool{B}\in G}$.
%We make $\mathcal{M}/_G$ a first order structure for the language 
%$\tau_{\ST}$, letting for a $\Delta_0$-formula $\phi$
%\begin{itemize}
%\item
%$R_\phi/_G([\tau_1]_G,\dots,[\tau_n]_G)$ if and only if $\Qp{R_\phi(1\tau_1,\dots,1\tau_n)}\in G$
%for $R_\phi$ a $\Delta_0$-property, 
%%\item
%%$f_\phi/_G([\tau_1]_G,\dots,[\tau_n]_G)=[\sigma]_G$ if and only if $\Qp{f_\phi(1\tau_1,\dots,1\tau_n)=\sigma}\in G$,
%\item
%similarly we handle the interpretation of the constant symbols of $\tau_{\ST}$.
%\end{itemize}

When $\bool{B}\in V$ is a 
$<\kappa$-cc complete boolean algebra, then  
$\Qp{\check{\kappa}\text{ is a regular cardinal}}=1_{\bool{B}}$. Therefore 
$H_{\check{\kappa}}^\bool{B}$ is 
a canonical set of $\bool{B}$-names which describes the $H_\kappa$ of a generic extension of $V$ by $\bool{B}$.

The choice to work with $H_{\dot{\kappa}}^\bool{B}$, instead of $V^\bool{B}$, is motivated by the fact that the 
former is a set definable in $V$ using the parameters 
$\bool{B}$ and $\dot{\kappa}$, while the latter is just a definable class in the parameter $\bool{B}$.

Having defined the structures we will be interested in (the structures 
$H_{\dot{\kappa}}^\bool{B}/_G$) we now turn to the definition of the relevant morphisms
between them.

\begin{definition}
Given $i:\bool{B}\to\bool{C}$ complete homomorphism of complete boolean algebras,
$i$ extends to a map $\hat{i}:V^{\bool{B}}\to V^{\bool{C}}$ defined by transfinite recursion by
\[
\hat{i}(\tau)=\bp{\ap{\hat{i}(\sigma),i(b)}: \,\ap{\sigma,b}\in\tau}.
\] 
Let $\dot{\kappa}\in V^{\bool{B}}$, $\dot{\delta}\in V^{\bool{C}}$ be such that $H_{\dot{\kappa}}^\bool{B}$, $H_{\dot{\delta}}^\bool{C}$ are 
well defined according to Def. \ref{def:HkappaB}, and $i\restriction H_{\dot{\kappa}}^\bool{B}\subseteq H_{\dot{\delta}}^\bool{C}$.

Given $\tau_1,\dots,\tau_n\in H_{\dot{\kappa}}^\bool{B}$,
$\phi(\tau_1,\dots,\tau_n)$ is generically absolute for $i$, $H_{\dot{\kappa}}^\bool{B}$, $H_{\dot{\delta}}^\bool{C}$ if 
\[
i(\Qp{\phi(\tau_1,\dots,\tau_n}_\bool{B}^{H_{\dot{\kappa}}^\bool{B}})=
\Qp{\phi(\hat{i}(\tau_1),\dots,\hat{i}(\tau_n)}^{H_{\dot{\delta}}^\bool{C}}_\bool{C}.
\]

$i$ is a boolean $\Sigma_n$-embedding of $H_{\dot{\kappa}}^\bool{B}$ into $H_{\dot{\delta}}^\bool{C}$ if all $\Sigma_n$-formulae for $\tau_{\ST}$ 
are generically absolute for $i$, $H_{\dot{\kappa}}^\bool{B}$, $H_{\dot{\delta}}^\bool{C}$ ($i$ is a boolean embedding if it preserves only atomic formulae).
\end{definition}
%$\phi(\tau_1,\dots,\tau_n)$ is generically absolute for $i$ if the above holds for $V^{\bool{B}}$, $V^{\bool{C}}$ in the palce

%It is not hard to check that $\Delta_1$-properties\footnote{I.e. properties which are the extension at the same time of a $\Pi_1$-formula and of 
%a $\Sigma_1$-formula provably in a suitable fragment of $\ZFC$.} are generically absolute between  for $i$, $H_{\dot{\kappa}}^\bool{B}$, $H_{\dot{\delta}}^\bool{C}$ whenever $i:\bool{B}\to\bool{C}$ is a complete homomorphism and  
%$i\restriction H_{\dot{\kappa}}^\bool{B}\subseteq H_{\dot{\delta}}^\bool{C}$
%(see for example \cite[Prop. 4.1.2]{VIAAUDSTEBOOK});
%but it can be argued that $\Sigma_1$-properties in real parameters are also generically absolute. 

We now prove the following Lemma:

\begin{lemma}\label{lem:Cohengen}
%Let $\phi(x,y)$ be a $\Delta_0$-formula.
Let $i:\bool{B}\to\bool{C}$ be a complete homomorphism such that 
$\Qp{\dot{\kappa}\text{ is a regular cardinal}}_{\bool{B}} =1_{\bool{B}}$, 
$\Qp{\dot{\delta}\text{ is a regular cardinal}}_{\bool{C}} =1_{\bool{C}}$, and
$\Qp{\hat{i}(\dot{\kappa})\leq\dot{\delta}}_{\bool{C}}=1_{\bool{C}}$.
Then:
\begin{enumerate} 
\item \label{lem:Cohengen-1}
$i$ is a boolean embedding. 
\item \label{lem:Cohengen-2}
For any $H\in St(\bool{C})$, letting $G\in St(\bool{B})$
be $k^{-1}[H]$,
the map 
\begin{align*}
\hat{i}/_H:&H_{\dot{\kappa}}^{\bool{B}}/_G\to H_{\dot{\delta}}^{\bool{C}}/_H\\
&[\tau]_G\mapsto [\hat{i}(\tau)]_H
\end{align*}
is a $\tau_{\ST}$-morphism.
\item \label{lem:Cohengen-3}
Assume further that $\dot{\kappa}=\omega_1^{\bool{B}}$. Then
$i$ is 
$\Sigma_1$-elementary.
\end{enumerate}
\end{lemma}
The first two parts are a straightforward consequence of the preservation of $\Delta_1$-properties through 
forcing extensions, while the third part
follows from Shoenfield's absolutenss, given that $\Sigma_1$-properties in real parameters corresponds to 
$\Sigma^1_2$-properties and any element of $H_{\omega_1}$ is coded in an absolute manner by a real. 
Let us however give an explicit proof in the set-up we built so far so to make the reader acquainted with it.

\begin{proof}
\emph{}

\begin{enumerate}

\item
Given a $\Delta_0$-formula $\phi$ for the signature $\bp{\in,\subseteq,=}$, we must show that 
\[
i(\Qp{\phi(\tau_1,\dots,\tau_n}_\bool{B}^{H_{\dot{\kappa}}^\bool{B}})=
\Qp{\phi(\hat{i}(\tau_1),\dots,\hat{i}(\tau_n)}^{H_{\dot{\delta}}^\bool{C}}_\bool{C}.
\]
	We prove the result by induction on the number of bounded quantifiers in $\phi$. For atomic formulas 
	$\psi$ (either $x = y$ or $x \in y$), we proceed by further induction on the rank of 
	$\tau_1$, $\tau_2$.
	
	\[
	\begin{split}
		i\cp{ \Qp{\tau_1 \in \tau_2}_{\bool{B}} } &= 
		i\cp{ \bigvee \bp{\tau_2(\dot{a}) \wedge \Qp{\tau_1 = 
		\dot{a}}_{\bool{B}} : \dot{a} \in \dom(\tau_2)}} \\
		&= \bigvee \bp{ i\cp{\tau_2(\dot{a})} \wedge i\cp{ \Qp{\tau_1 = 
		\dot{a}}_{\bool{B}}}: {\dot{a} \in \dom(\tau_2)}} \\
		&= \bigvee\bp{ i\cp{\tau_2(\dot{a})} \wedge \Qp{\hat{\imath}(\tau_1) = 
		\hat{\imath}(\dot{a})}_{\bool{C}} : {\dot{a} \in \dom(\tau_2)}}\\
		&= \Qp{\hat{\imath}(\tau_1) \in \hat{\imath}(\tau_2)}_{\bool{C}} \\
		i\cp{ \Qp{\tau_1 \subseteq \tau_2}_{\bool{B}} } &= 
		i\cp{ \bigwedge \bp{ \tau_1(\dot{a}) \rightarrow \Qp{\dot{a} \in 
		\tau_2}_{\bool{B}}: {\dot{a} \in \dom(\tau_1)}}} \\	
		&= \bigwedge \bp{ i\cp{\tau_1(\dot{a})}\rightarrow 
		i\cp{ \Qp{\dot{a} \in \tau_2}_{\bool{B}} }: {\dot{a} \in \dom(\tau_1)}} \\
		&= \bigwedge \bp{ i\cp{\tau_1(\dot{a})} \rightarrow 
		\Qp{\hat{\imath}(\dot{a}) \in \hat{\imath}(\tau_2)}_{\bool{C}}: {\dot{a} \in \dom(\tau_1)}} \\
		&=\Qp{\hat{\imath}(\tau_1) \subseteq \hat{\imath}(\tau_2)}_{\bool{C}}. 
	\end{split}
	\]
	
	We used the inductive hypothesis in the last row of each case. Since 
	$\Qp{\tau_1 = \tau_2} = \Qp{\tau_1 \subseteq \tau_2} \wedge \Qp{\tau_2 
	\subseteq \tau_1}$, the proof for $\psi$ atomic $\bp{\in,\subseteq,=}$-formula is complete. 
		
	The induction step for boolean connectives is left to the reader. 
	Suppose now that $\psi = \exists x \in y ~ \phi$ is a 
	$\Delta_0$-formula for the signature  $\bp{\in,\subseteq,=}$
	and the inductive hypothesis holds for $\phi$. Then
	\[
	\begin{split}
	    i&\cp{\Qp{\exists x \in \tau_1 \phi(x,\tau_1,\ldots,\tau_n)}_{\bool{B}}} \\
		& = \bigvee \bp{  i\cp{\tau_1(\dot{a})} \wedge 
		i\cp{\Qp{\phi(\dot{a},\tau_1,\ldots,\tau_n)}_{\bool{B}}}: {\dot{a} \in \dom(\tau_1)}} \\
		& = \bigvee \bp{ i(\tau_1(\dot{a})) \wedge 
		\Qp{\phi\cp{\hat{\imath}(\dot{a}),
		\hat{\imath}(\tau_1),\ldots,\hat{\imath}(\tau_n)}}_{\bool{C}} : {\dot{a} \in \dom(\tau_1)}} \\
		& = \Qp{\exists x \in \hat{\imath}(\tau_1) ~ 
		\phi\cp{x,\hat{\imath}(\tau_1),\ldots,\hat{\imath}(\tau_n)}}_{\bool{C}}
	\end{split}
	\]
	
\item Immediate by the forcing theorem for $H_{\dot{\kappa}}^{\bool{B}}$ and $H_{\dot{\delta}}^{\bool{C}}$, and 
the previous item.
	
%
%	Furthermore, if $\psi = \exists x ~ \phi$ is a $\Sigma_1$ formula, by the Maximum Principle there exists a $\dot{a} \in V^{\bool{B}}$ such that 
%	$\Qp{\exists x \phi(x,\tau_1,\ldots,\tau_n)}_{\bool{B}} = 
%	\Qp{\phi(\dot{a},\tau_1,\ldots,\tau_n)}_{\bool{B}}$ hence
%	\[
%	\begin{split}
%		i\cp{\Qp{\exists x \phi(x,\tau_1,\ldots,\tau_n)}_{\bool{B}}} &= 
%		i\cp{\Qp{\phi(\dot{a},\tau_1,\ldots,\tau_n)}_{\bool{B}}} \\
%		&= \Qp{\phi\cp{\hat{\imath}(\dot{a}),\hat{\imath}(\tau_1),
%		\ldots,\hat{\imath}(\tau_n)}}_{\bool{C}} \\
%		&\leq \Qp{\exists x\phi\cp{x,\hat{\imath}(\tau_1),\ldots,\hat{\imath}(\tau_n)}}_{\bool{C}}
%	\end{split}
%	\]
%	Thus, if $\phi$ is a $\Delta_1$ formula, $\phi$ and $\neg \phi$ are both $\Sigma_1$, 
%	hence the above inequality holds and also
%	\[
%	\begin{split}
%		i\cp{\Qp{\phi(\tau_1,\ldots,\tau_n)}_{\bool{B}}} &= \neg i\cp{\Qp{\neg\phi(\tau_1,\ldots,\tau_n)}_{\bool{B}}} \\
%		&\geq \neg \Qp{\neg \phi\cp{\hat{\imath}(\tau_1),\ldots,\hat{\imath}(\tau_n)}}_{\bool{C}} \\
%		&= \Qp{\phi\cp{\hat{\imath}(\tau_1),\ldots,\hat{\imath}(\tau_n)}}_{\bool{C}}, \\
%	\end{split}
%	\]
%	concluding the proof.

\item
Assume $\phi(x,x_1,\dots,x_n)$ is a $\Sigma_0$-formula for $\bp{\in,\subseteq,=}$ and
$(\tau_1,\dots,\tau_n)$ is a tuple in $H_{\dot{\omega_1}}^{\bool{B}}$ such that
\[
\Qp{\exists x\phi(x,\hat{i}(\tau_1),\dots,\hat{i}(\tau_n))}^{H_{\dot{\delta}}^\bool{C}}_\bool{C}\geq i(b).
\]
It suffices to show that 
\[
\Qp{\exists x\phi(x,\tau_1,\dots,\tau_n)}^{H_{\dot{\omega}_1}^\bool{B}}_\bool{B}\geq b.
\] 
Fix $G$ $V$-generic for $\bool{B}$ such that $b\in G$. Work in $V[G]$.

Remark that $\bool{C}/_{i[G]}$ is a boolean algebra in $V[G]$ whose elements are the equivalence classes
$[c]_G=\bp{d\in\bool{C}: \Qp{d=c}\in G}$ and with quotient boolean operations (it is actually a complete boolean algebra in $V[G]$, 
but this does not matter here).

Let $P$ be in $V[G]$ the forcing notion $(\bool{C}/_{i[G]})^+$ given by the positive elements of the boolean algebra.

Pick a model $M\in V[G]$ such that 
$M\prec (H_{|P|^+})^{V[G]}$, $M$ is countable in $V[G]$,  and
$P,a_1=(\tau_1)_G,\dots,a_n=(\tau_n)_G\in M$. 
Let $\pi_M:M\to N$ be its transitive collapse
and $Q=\pi_M(P)$. Notice also that 
$\pi_M(a_i)=a_i$ for $i=1,\dots,n$ since $a_i\in H_{\omega_1}^{V[G]}\cap M$ and this latter set is transitive since $\omega\subseteq M$.
Since $\pi_M$ is an isomorphism of $M$ with $N$,
\[
N\models(\Vdash_{Q}\exists x\phi(x,a_1,\dots,a_n)).
\] 
Now let $H\in V[G]$ be $N$-generic for $Q$ ($H$ exists in $V[G]$ since $N$ is countable in $V[G]$),
then, by Cohen's fundamental theorem of forcing applied in $V[G]$ to $N$ and $Q$
(since $N$ is a countable transitive model of a large enough initial fragment of $\ZFC$), 
we have that $N[H]\models\exists x\phi(x,a_1,\dots,a_n)$.
So we can pick $a\in N[H]$ such that $N[H]\models\phi(a,a_1,\dots,a_n)$. 
Since $N,H\in (H_{\aleph_1})^{V[G]}$ are countable in $V[G]$, $N[H]$ is countable and transitive; hence we have that
$V[G]$ models that $N[H]\in H_{\omega_1}^{V[G]}$, and thus $V[G]$ models that 
$a$ as well belongs to $H_{\omega_1}^{V[G]}$.
Since $\phi(x,x_1,\dots,x_n)$ is a $\Sigma_0$-formula, 
$V[G]$ models that $\phi(a,a_1,\dots,a_n)$ is absolute between the
transitive sets $N[H]\subset H_{\omega_1}^{V[G]}$
to which $a,a_1,\dots,a_n$ belong. In particular
$a$ witnesses in $V[G]$ that $H_{\omega_1}^{V[G]}\models\exists x\phi(x,a_1,\dots,a_n)$. 

Since the argument can be repeated for any $G$ $V$-generic for $\bool{B}$ with $b\in G$, 
we conclude that 
\[
H_{\omega_1}^{V[K]}\models\exists x\phi(x,(\tau_1)_K,\dots,(\tau_n)_K)
\]
for any $K$ $V$-generic for $\bool{B}$ with $b\in K$.

By the maximum principle and the forcing theorem applied in $V$ to $\bool{B}$, 
we can find $\sigma\in V^{\bool{B}}$ such that
\[
\Qp{(\sigma\text{ is hereditarily countable})\wedge \phi(\sigma,\tau_1,\dots,\tau_n)}^{V^{\bool{B}}}_\bool{B}\geq b.
\]
Since $\sigma$ is a $\bool{B}$-name for an hereditarily countable set according to $b$, it is decided by countably many antichains
below\footnote{For example let $\dot{R}$ be a canonical name for a binary relation on $\omega$ coding the transitive closure of $\bp{\sigma_G}$
for any $G$ $V$-generic for $\bool{B}$ with $b\in G$. Then for any such $G$ the transitive collapse of 
$\dot{R}_G$ is the transitive closure of $\bp{\sigma_G}$, and clearly 
$\dot{R}$ is decided by countably many antichains below $b$.} $b$.

In particular we can find in $V$, $\tau\in H_{\dot{\omega_1}}^{\bool{B}}$ such that $\Qp{\sigma=\tau}^{V^{\bool{B}}}_\bool{B}\geq b$.
We conclude that 
\[
\Qp{\phi(\tau,\tau_1,\dots,\tau_n)}^{V^{\bool{B}}}_\bool{B}\geq b.
\]
Since $\phi$ is a $\Delta_0$-formula
\[
\Qp{\phi(\tau,\tau_1,\dots,\tau_n)}^{H_{\dot{\omega_1}}^{\bool{B}}}=\Qp{\phi(\tau,\tau_1,\dots,\tau_n)}^{V^{\bool{B}}}_\bool{B}\geq b.
\]
The thesis follows.
\end{enumerate}
\end{proof}

\begin{definition}
The \emph{generic multiverse} $(\Omega(V),\to_{\Omega(V)})$ is the collection:
\[
\bp{H_{\dot{\kappa}}^{\bool{B}}/_G:\,
\Qp{\dot{\kappa}\text{ is a regular cardinal}}_{\bool{B}}=1_{\bool{B}},\,G\in St(\bool{B})}.
\]
Its morphism are the $\tau_{\ST}$-morphisms of the form $\hat{i}/_H:H_{\dot{\kappa}}^{\bool{B}}/_G\to H_{\dot{\delta}}^{\bool{C}}/_H$ 
for some complete homomorphism 
$i:\bool{B}\to\bool{C}$ with $H\in St(\bool{C})$, $G=f^{-1}[H]$, 
$\Qp{\hat{i}(\dot{\kappa})\leq\dot{\delta}}_{\bool{C}}=1_{\bool{C}}$.
\end{definition}

Notice\footnote{There can be morphisms $h:H_\kappa^{\bool{B}}/_G\to H_\delta^{\bool{C}}/_H$ which are not of the form
$\hat{f}/_H$ for some complete homomorphism $f:\bool{B}\to \bool{C}$, even in case 
$\bool{B}$ preserve the regularity of $\kappa$ and $\bool{C}$ the regularity of $\delta$. 
%However this occur only if $\kappa\geq\delta$ and further rather uncommon conditions are met by $\bool{B}$ and $\bool{C}$.
We do not spell out the details of such possibilities.} that $\Omega(V)$ is a definable class in $V$. $\Omega(V)$ is a formulation in the language of boolean valued models of the notion of generic multiverse.

This is the first result we want to bring forward:

\begin{theorem}\label{thm:omegaVmodcompl}
The family 
\[
\bp{H_{\dot{\omega_1}}^{\bool{B}}/_G:\,\bool{B}\text{ is a cba and }G\in St(\bool{B})}.
\]
is the class of generic structures of the generic multiverse $(\Omega(V),\to_{\Omega(V)})$.
\end{theorem}
%The first order theory with parameters for elements of $H_{\omega_1}^V$ 
%of the $\tau_{\ST}$-structure 
%$(H_{\omega_1}^V,R_\phi^V: R_\phi\in\tau_{\ST})$ in the signature $\tau_{\ST}\cup H_{\omega_1}$ 
%is the model companion of 
%$\ZFC_{\ST}+$\emph{ there exist class many Woodin cardinals} with respect to $(\Omega(V),\to_{\Omega(V)})$.  

\begin{proof}
By Lemma \ref{lem:Cohengen}(\ref{lem:Cohengen-3}), the $\tau_{\ST}$-models 
of type $H_{\dot{\omega_1}}^{\bool{B}}/_G$ are $\Sigma_1$-elementary in any of their superstructures in
$\Omega(V)$. By \cite[Cor 26.8 ]{JECH} any model of the form $H_{\dot{\kappa}}^{\bool{B}}$ is absorbed
by a model of the form $H_{\dot{\omega_1}}^{\bool{C}}$, where $\bool{C}$ is the boolean completion of 
$\Coll(\omega,H_{\dot{\kappa}}^{\bool{B}})$.

The thesis follows immediately.
\end{proof}

A natural questions arise: 
\begin{quote}
Is the $\tau_{\ST}$-theory 
$T$ of $H_{\omega_1}^V$ model complete?
\end{quote}
If this question has an affirmative answer, $T$ is the model companion of the $\tau_{\ST}$-theory of $V$: 
Levy's absoluteness allows to prove that
$H_{\omega_1}$ is $\Sigma_1$-elementary 
in $V$ with respect to $\tau_{\ST}$, hence the two structures have the same universal theory and we 
can apply Robinson's test to them.

%In the forthcoming \cite{VIAPAR} the second author and Parente show that 
%any model of t
%the answer to the second question is positive 
%(assuming large cardinal axioms). This is already quite interesting: it outlines that any set sized 
%$\tau_{\ST}$-model 
%of the theory of (an initial segment of) $V$ (obtained by whatever means model theory gives us) is in fact a substructure of
%a $\tau_{\ST}$-model of the theory of (an initial segment of) $V$ obtained by forcing.

In the next section we
argue that this question has a negative answer. 
Nonetheless in the forthcoming \cite{VIAPAR} the second author and Parente will show that
the models of type $H_{\dot{\omega_1}}^{\bool{B}}/_G$ are existentially closed for their own universal theory and 
contain a copy of any set sized model of the $\tau_{\ST}$-universal theory of $V$.

%To address the question of model companionship for $\ZFC$ we further expand the language of set theory, including predicates for universally Baire sets, in order to
%argue that Woodin's generic absoluteness results for this type of sets bring, as a byproduct,
%the model completeness of the $\tau_{\ST}$-theory of $H_{\omega_1}$ enriched now 
%with predicates for the universally Baire sets.

\section{Second order arithmetic and $\text{Th}(H_{\omega_1})$} 	\label{sec:nonmodcompletHomega1}

We define second order number theory as the first order theory of the structure
\[
(\mathcal{P}(\mathbb{N})\cup\mathbb{N},\in,\subseteq,=,\mathbb{N}).
\]

$\Pi^1_n$-sets (respectively $\Sigma^1_n$, $\Delta^1_n$)  are
the subsets of $\mathcal{P}(\mathbb{N})\equiv 2^{\mathbb{N}}$ 
defined by a $\Pi_n$-formula (respectively by a $\Sigma_n$-formula, at the same time 
by a $\Sigma_n$-formula and a $\Pi_n$-formula in the appropriate language), if the formula defining a set 
$A\subseteq (2^{\mathbb{N}})^n$ has some parameter
$\vec{r}\in (2^{\mathbb{N}})^{<\omega}$ we accordingly write that $A$ is $\Pi^1_n(\vec{r})$ 
(respectively $\Sigma^1_n(\vec{r})$, $\Delta^1_n(\vec{r})$).
if the formula defining a set 
$A\subseteq (2^{\mathbb{N}})^n$ uses an extra-predicate symbol $B\subseteq (2^\omega)^k$ we write that
$A$ is $\Pi^1_n(B)$ (respectively $\Sigma^1_n(B)$, $\Delta^1_n(B)$).

$A\subseteq (2^{\mathbb{N}})^N$ is projective if it is defined by some $\Pi^1_n$-property for some $n$.
Similarly we define the notion of being projective in $\vec{r}\in (2^{\mathbb{N}})^{<\omega}$ or 
$B\subseteq  (2^\omega)^k$.
\begin{remark}
$A\subseteq (2^{\mathbb{N}})^k$ is projective if and only if it is obtained by a Borel subset of $(2^{\mathbb{N}})^m$
by successive applications of the operations of projection on one coordinate and complementation.
\end{remark}

\begin{definition}
Given $a\in H_{\omega_1}$, $r\in 2^{\mathbb{N}}$ codes $a$, if (modulo a recursive bijection
of $\mathbb{N}$ with $\mathbb{N}^2$), $r$ codes a well-founded extensional relation on 
$\mathbb{N}$
whose transitive collapse is the transitive closure of $\{a\}$.

\begin{itemize}
\item
 $\mathrm{Cod}:2^{\mathbb{N}}\to H_{\omega_1}$ is the map assigning $a$ to $r$ if and only if $r$ codes $a$
and assigning the emptyset to $r$ otherwise.
\item
$\mathrm{WFE}$ is the set of $r\in 2^{\mathbb{N}}$ which (modulo a recursive bijection
of $\mathbb{N}$ with $\mathbb{N}^2$) are a well founded extensional relation.
\end{itemize}
\end{definition}

The following are well known facts\footnote{See \cite[Section 25]{JECH} and in particular the statement and proof of Lemma 25.25, which contains all ideas on which one can elaborate to draw the conclusions below.}.
\begin{remark} 
The map $\mathrm{Cod}$ is defined by a $\ZFC^-$-provably $\Delta_1$-property  (with no parameters)
over $H_{\omega_1}$and is surjective. Moreover $\mathrm{WFE}$ is a $\Pi^1_1$-subset of $2^{\mathbb{N}}$. 
%(in particular a Universally Baire set).
Therefore if $N\in H_{\omega_1}$ is a transitive model of $ZFC$, 
$N$ computes correctly $\mathrm{Cod}$ and $\mathrm{WFE}$, i.e. $\mathrm{Cod}^N=\mathrm{Cod}\cap N$ and 
$\mathrm{WFE}^N=\mathrm{WFE}\cap N$.
%(Elaborate on the proof of \cite[Lemma 25.25]{JECH}).
\end{remark}

\begin{lemma}
Assume $A\subseteq 2^{\mathbb{N}}$ is $\Sigma^1_{n+1}$. Then
$A$ is  $\Sigma_{n}$-definable in $H_{\omega_1}$ in the language $\tau_{\ST}$.\qed
\end{lemma}
%\begin{proof}
%See \cite[Lemma 25.25 ]{JECH} for a proof of the case $n=1$. 
%The generalization to arbitrary $n$ is left to the reader.
%\end{proof}

\begin{lemma}
Assume $A$ is $\Sigma_n$-definable in $H_{\omega_1}$ in the language $\tau_{\ST}$. Then
$A=\textrm{Cod}[\textrm{Cod}^{-1}[A]]$, and $\textrm{Cod}^{-1}[A]$ is
$\Sigma^1_{n+1}$.
\end{lemma}

We can now easily conclude the following:
\begin{theorem}\label{thm:nonmodcompHomega1}
The $\tau_{\ST}$-theory 
of $H_{\omega_1}$ is not model complete.
\end{theorem}
\begin{proof}
For all $n$ there is some $A_n\in \Sigma^1_{n+1}\setminus\Pi^1_n$ (see for a proof \cite[Thm. 22.4]{kechris:descriptive}).
Therefore $A_2$ is $\Sigma_2$-definable but not $\Pi_1$-definable in $H_{\omega_1}$.
Consequently, Robinson test fails, and $T$ is not model complete.
\end{proof}

\section{Model completeness for set theory with predicates for universally Baire sets}\label{sec:modcompanUBpred}

\begin{definition}
Let $(V,\in)$ be a model of $\ZFC$ and $N\subseteq V$ be a transitive class (or set) which is a model of 
$\ZFC^-$.
$\mathcal{A}\subseteq$ is $N$-closed
if whenever
$B\subseteq (2^\omega)^k$ is such that for some 
$\in$-formula $\phi(x_0,\dots,x_n)$
\[
B=\bp{(r_0,\dots,r_{k-1})\in (2^\omega)^k:\, (N,\in,A_0,\dots,A_{n-k})\models\phi(r_0,\dots,r_{k-1},A_0,\dots,A_{n-k})}
\]
with $A_0,\dots,A_{n-k}\in\mathcal{A}$, we have that $B\in\mathcal{A}$.
\end{definition}

\begin{theorem}\label{thm:modcompanHomega1}
Let $(V,\in)$ be a model of 
$\ZFC$, %\emph{there are class many Woodin cardinals}, 
and assume
$\mathcal{A}$ is $H_{\omega_1}$-closed.

Let $\tau_{\mathcal{A}}=\tau_{\ST}\cup\mathcal{A}$.
Then the $\tau_{\mathcal{A}}$-theory of $H_{\omega_1}$ 
is model complete and is the
model companion of the $\tau_{\mathcal{A}}$-theory of $V$.
\end{theorem}

The proof is rather straightforward but needs a slight generalization of Levy's absoluteness theorem which we
state and prove rightaway:

\begin{lemma}\label{lem:levyabsHkappa+}
Let $\kappa$ be an infinite  cardinal
and $\mathcal{A}$ be any family of subsets of 
$\bigcup_{n\in\omega}\pow{\kappa}^n$.
Let $\tau_{\mathcal{A}}=\tau_{\ST}\cup\mathcal{A}$.

Then:
\[
(H_{\kappa^+}^V,\tau_{\mathcal{A}}^V)\prec_{\Sigma_1}
(V,\tau_{\mathcal{A}}^{V}).
\]
\end{lemma}

\begin{proof}
Assume for some $\tau_{\mathcal{A}}$-formula
$\phi(\vec{x},y)$ without quantifiers\footnote{A quantifier free $\tau_{A_1,\dots,A_k}$-formula is a boolean
combination of atomic $\tau_{\ST}$-formulae with formulae of type $A_j(\vec{x})$.
For example $\exists x\in y A(y)$ is not a quantifier free $\tau_{\ST}$-formula, and is actually equivalent to the 
$\Sigma_1$-formula
$\exists x(x\in y)\wedge A(y)$.}
and 
$\vec{a}\in H_{\kappa^+}$
\[
(V,\tau_{\mathcal{A}}^{V})\models\exists y\phi(\vec{a},y).
\]
Let $\alpha>\kappa$ be large enough  so that for some $b\in V_\alpha$
\[
(V,\tau_{\mathcal{A}}^{V})\models\phi(\vec{a},b).
\]
Then
\[
(V_\alpha,\tau_{\mathcal{A}}^{V})\models\phi(\vec{a},b).
\]
Let $A_1,\dots,A_k$ be the subsets of $\pow{\kappa}^{i_k}$ which are the predicates mentioned in $\phi$. 
By the downward Lowenheim-Skolem theorem, we  can find
$X\subseteq V_\alpha$ which is the domain of a $\tau_{A_1,\dots,A_k}$-elementary substructure of
\[
(V_\alpha,\tau_{\ST},A_1,\dots,A_k)
\]
such that $X$ is a set of size $\kappa$ containing $\kappa$ and such that
$A_1,\dots,A_k,\kappa,b,\vec{a}\in X$. 
Since $|X|=\kappa\subseteq X$, a standard argument shows that 
$H_{\kappa^+}\cap X$ is a transitive set, and that $\kappa^+$ is the least ordinal in
$X$ which is not contained in $X$.
Let $M$ be the transitive collapse of $X$ via the Mostowski collapsing map $\pi_X$.

We have that
the first ordinal moved by $\pi_X$ is $\kappa^+$ and
$\pi_X$ is the identity on $H_{\kappa^+}\cap X$. Therefore $\pi_X(a)=a$
for all 
$a \in H_{\kappa^+}\cap X$.
Moreover for $A\subseteq \pow{\kappa}^n$ in $X$
\begin{equation}\label{eqn:piXidonpowkappa}
\pi_X(A)=A\cap M.
\end{equation}
We prove equation (\ref{eqn:piXidonpowkappa}):
\begin{proof}
%The inclusion $\pi_X(A)\subseteq A\cap M$ is clear
%since $\pi_X$ is the identity on $\pow{\kappa}\cap X$
%and $A\subseteq\pow{\kappa}$.
%Now assume $Z\in A\cap M$.
%Then $Z$ belongs to $M\cap V_{\kappa+1}$.
Since 
$X\cap V_{\kappa+1}\subseteq X\cap H_{\kappa^+}$,
$\pi_X$ is 
the identity on $X\cap H_{\kappa^+}$, and
$A\subseteq \pow{\kappa}\subseteq V_{\kappa+1}$,
we get that
\[
\pi_X(A)=\pi_X[A\cap X]=\pi_X[A\cap X\cap V_{\kappa+1}]=A\cap M\cap V_{\kappa+1}=A\cap M.
\]
\end{proof}
It suffices now to show that
\begin{equation}\label{eqn:keyeqlevabs}
(M,\tau_{\ST}^V,\pi_X(A_1),\dots,\pi_X(A_k))\sqsubseteq (H_{\kappa^+},\tau_{\ST}^V,A_1,\dots,A_k).
\end{equation}
Assume \ref{eqn:keyeqlevabs} holds; since $\pi_X$ is an isomorphism and $\pi_X(A_j)=\pi_X[A_j\cap X]$, we
get that 
\[
(M,\tau_{\ST}^V,\pi_X(A_1),\dots,\pi_X(A_k))\models\phi(\pi_X(b),\vec{a})
\]
since 
\[
(X,\tau_{\ST}^V,A_1\cap X,\dots,A_k\cap X)\models\phi(b,\vec{a}).
\]
By (\ref{eqn:keyeqlevabs}) we get that 
\[
(H_{\kappa^+},\tau_{\ST}^V,\pi_X(A_1),\dots,\pi_X(A_k))\models\phi(\pi_X(b),\vec{a})
\]
and we are done.

We prove (\ref{eqn:keyeqlevabs}):
since $M$ is transitive, any atomic $\tau_{\ST}$-formula (i.e. any $\Delta_0$-property)
holds true in $M$ if and only if it holds in $H_{\kappa^+}$.
It remains to argue that the same occurs for the $\tau_{\mathcal{A}}$-formulae of type $A_j(x)$, i.e. that
$A_j\cap M=\pi_X(A_j)$ for all $j=1,\dots,n$; which is the case
by (\ref{eqn:piXidonpowkappa}).
\end{proof}

\begin{remark}
Key to the proof is the fact that subsets of $\kappa$
have bounded rank below $\kappa^+$.
If $A\subseteq H_{\kappa^+}$ has elements of 
unbounded rank, the equality $\pi_X(A)=A\cap M$ 
may fail: for example
if $A=H_{\kappa^+}$, $\pi_X(A)=H_{\kappa^+}\cap X$
while $A\cap M=M$.
This shows that (\ref{eqn:keyeqlevabs}) fails for this choice 
of $A$.
\end{remark}

We can now prove Theorem \ref{thm:modcompanHomega1}.

\begin{proof}
Let $T$ be the $\tau_{\mathcal{A}}$-theory of $V$ and $T^*$ be the 
$\tau_{\mathcal{A}}$-theory of $H_{\omega_1}$.

By the version of Levy's absoluteness Lemma we just proved
\[
(H_{\omega_1},\tau_{\mathcal{A}}^V)\prec_1(V,\tau_{\mathcal{A}}^V),
\]
hence the two structures share the same $\Pi_1$-theory.
Therefore (by the standard characterization of model companionship) 
it suffices to prove that $T^*$ is model complete.

By Robinson's test, 
it suffices to show that any existential $\tau_{\mathcal{A}}$-formula is 
$T^*$-equivalent to a universal $\tau_{\mathcal{A}}$-formula.

Let $A_1,\dots,A_k$ be the predicates in $\mathcal{A}$ appearing in $\phi$.

%Since $\phi$ is existential and $\mathcal{A}\subseteq \pow{\pow{\omega^{<\omega}}}$ 
%(by \cite[Lemma LEVABS]{VIATAMSTI}), 
%for any $r_1,\dots,r_n\in 2^\omega$
%\[
%(V,\tau_\mathcal{A}^V)\models\phi(\Cod_\omega(r_1),\dots,\Cod_\omega(r_n))
%\]
%if and only if so does
%\[
%(H_{\omega_1},\tau_{\ST}^V,A_1,\dots,A_k).
%\]
%Find $\psi(x_1,\dots,x_n,y_1,\dots,y_k)$ $\in$-formula such that
%\[
%(V,\tau_{\mathcal{A}}^V)\models
%\forall x_1,\dots,x_n\,(\phi(x_1,\dots,x_n)\leftrightarrow\psi(x_1,\dots,x_n,A_1,\dots,A_k)).
%\] 
Let%\footnote{See \cite[Def. 2.2]{VIATAMSTI} for a definition of $\Cod_\omega$.} 
\[
B=\bp{(r_1,\dots,r_n)\in (2^\omega)^n: 
(H_{\omega_1},\tau_{\ST}^V,A_1,\dots,A_k)\models \phi(\Cod(r_1),\dots,\Cod_\omega(r_n))}.
\]
Then $B$ belongs to $\mathcal{A}$, since $\mathcal{A}$ is  $H_{\omega_1}$-closed. 
Now for any $a_1,\dots,a_n\in H_{\omega_1}$:
\[
(H_{\omega_1},\tau_{\mathcal{A}}^V)\models \phi(a_1,\dots,a_n)
\]
%\center{ if and only if }
%\[
%(V,\tau_{\mathcal{A}}^V)\models \phi(a_1,\dots,a_n)
%\]
%\center{ if and only if }
%\[
%(L(a_1,\dots,a_n,A_1,\dots,A_k),\tau_{\ST}^V,A_1,\dots,A_k)\models \phi(a_1,\dots,a_n)
%\]
%\center{ if and only if }
%\[
%\forall r_1\dots r_n \bigwedge_{i=1}^n\Cod_\omega(r_i)=a_i\rightarrow
%[(L(r_1,\dots,r_n,A_1,\dots,A_k),\tau_{\ST}^V,A_1,\dots,A_k)\models 
%\phi(\Cod_\omega(r_1),\dots,\Cod_\omega(r_n))]
%\]
\center{ if and only if }
\[
\forall r_1\dots r_n \qp{(\bigwedge_{i=1}^n\Cod(r_i)=a_i)\rightarrow B(r_1,\dots,r_n)}.
\]

%\begin{align*}
%(H_{\omega_1},\tau_{\mathcal{A}}^V)\models \phi(a_1,\dots,a_n)\\
%\text{ if and only if }\\
%(V,\tau_{\mathcal{A}}^V)\models \phi(a_1,\dots,a_n)\\
%\text{ if and only if }\\
%(L(a_1,\dots,a_n,A_1,\dots,A_k),\tau_{\ST}^V,A_1,\dots,A_k)\models \phi(a_1,\dots,a_n)\\
%\text{ if and only if }\\
%\forall r_1\dots r_n \bigwedge_{i=1}^n\Cod_\omega(r_i)=a_i\rightarrow
%[(L(r_1,\dots,r_n,A_1,\dots,A_k),\tau_{\ST}^V,A_1,\dots,A_k)\models 
%\phi(\Cod_\omega(r_1),\dots,\Cod_\omega(r_n))]\\
%\text{ if and only if }\\
%\forall r_1\dots r_n \bigwedge_{i=1}^n\Cod_\omega(r_i)=a_i\rightarrow B(r_1,\dots,r_n).
%\end{align*}

%\begin{align*}
%&(H_{\omega_1},\tau_{\mathcal{A}}^V)\models \phi(a_1,\dots,a_n)}&\\
%&\text{ if and only if }&\\
%%&=\bp{(a_1,\dots,a_n)\in (H_{\omega_1})^n:\, 
%&(V,\tau_{\mathcal{A}}^V)\models \phi(a_1,\dots,a_n)%}=\\
%&=\bp{(a_1,\dots,a_n)\in (H_{\omega_1})^n: \,
%(L(a_1,\dots,a_n,A_1,\dots,A_k),\tau_{\ST}^V,A_1,\dots,A_k)\models \phi(a_1,\dots,a_n)}=\\
%&=\bp{(\Cod_\omega(r_1),\dots,\Cod_\omega(r_n)):\,
%r_1,\dots,r_n\in 2^\omega \text{ and }
%(L(r_1,\dots,r_n,A_1,\dots,A_k),\tau_{\ST}^V,A_1,\dots,A_k)\models 
%\phi(\Cod_\omega(r_1),\dots,\Cod_\omega(r_n))}=\\
%&=\bp{(\Cod_\omega(r_1),\dots,\Cod_\omega(r_n)):\,(r_1,\dots,r_n)\in B}.
%\end{align*}

This yields that
\[
T^*\vdash 
\forall x_1,\dots,x_n\,\qp{(\phi(x_1,\dots,x_n)\leftrightarrow\theta_\phi(x_1,\dots,x_n))}.
\]
where $\theta_\phi(x_1,\dots,x_n)$ is the $\Pi_1$-formula in the predicate $B\in\mathcal{A}$
\[
\forall y_1,\dots,y_n\,[(\bigwedge_{i=1}^n x_i=\Cod(y_i))\rightarrow B(y_1,\dots,y_n)].
\]
\end{proof}

Theorem~\ref{thm:modcompanHomega1} has a rather straightforward proof which 
amounts to a (slightly) disguised addition of atomic predicates (i.e. those representing the elements of $\mathcal{A}$)
which interpret
the definable subsets of $H_{\omega_1}$. But the point we want to make is that assuming large cardinals 
the universally Baire sets give a very large sample of projectively closed families which are quite ``simple'', 
hence it is natural to consider elements of these families
as atomic predicates. 
%%The exact definition of what is meant by a ``simple'' subset of $2^\omega$ 
%is captured by the notion of universally Baire set.

Given a topological space $(X,\tau)$, $A\subseteq X$ is nowhere dense if its closure has a dense complement,
meager if it is the countable union of nowhere dense sets, with the Baire property if it 
has meager symmetric difference with
an open set.

\begin{definition}
(Feng, Magidor, Woodin) $A\subseteq (2^{\mathbb{N}})^k$ is \emph{universally Baire} 
if for every compact Hausdorff space $X$ and
every continuous $f:X\to (2^{\mathbb{N}})^k$ we have that $f^{-1}[A]$ has the Baire property in $X$.

$\bool{UB}$ denotes the family of universally Baire subsets of $(2^{\mathbb{N}})^k$ for some $k$.
\end{definition}

\begin{example}
Given a model $(V,\in)$ of $\ZFC+$\emph{there are class many Woodin cardinals},
simple examples of $H_{\omega_1}$-closed families are:
\begin{enumerate}
\item
All subsets of $\pow{\omega^k}$ for $k\in\mathbb{N}$ (this is trivially true with no large cardinal assumptions).
\item
$\UB^V$, i.e. the family of \emph{all} universally Baire sets of $V$ 
\cite[Thm. 3.3.3, Thm. 3.3.5, Thm. 3.3.6, Thm. 3.3.13, Thm. 3.3.14]{STATLARSON}.
\item
The family of subsets of $\pow{\omega^k}$ for $k\in\mathbb{N}$ definable in $(L(\mathbb{R}),\in)$
among which the projective sets 
\cite[Thm. 3.3.3, Thm. 3.3.5, Thm. 3.3.6, Thm. 3.3.9, Thm. 3.3.13, Thm. 3.3.14]{STATLARSON}.
\item The family $\UB\cap X$ for some $X\prec V_\theta$ with $\theta>2^\omega$.
%\item The family of lightface definable projective sets of reals.
\end{enumerate}
%It is also not hard to check that $L(\UB)$ is constructibly closed also if
%$\maxUB$ fails, but there are class many Woodin cardinals.
\end{example}

\begin{corollary}\label{thm:mainthm1}
Assume $(V,\in)$ models 
\[
\ZFC+\text{there are class many Woodins which are inaccessible limit of Woodins}.
\]
Let $\psi$ be a 
$\Pi_2$-sentence for $\tau_{\UB}=\tau_{\ST}\cup\UB$. TFAE:
\begin{enumerate}
\item 
$\psi$ is in the $\tau_{\UB}$-theory of $H_{\omega_1}$.
\item
$\psi$ is in the model companion of the $\tau_{\UB}$-theory of $V$.
\item
$\psi$ is consistent with the universal
fragment of the $\tau_{\UB}$-theory of $V$.
\item
$V$ models that some partial order $P$ forces $\psi^{H_{\omega_1}}$.
\end{enumerate}
\end{corollary}
\begin{proof}
The equivalence of the first and fourth items follow from \cite[Cor. 3.1.10]{STATLARSON}.
The equivalence of the first and second item follows from Thm. \ref{thm:modcompanHomega1}.

The equivalence of the second and third item is a purely model theoretic fact, 
we give the details for the sake of completeness:

Assume $T$ is a complete theory admitting a model companion $T^*$.
%The $\Pi_2$-sentences holding in all the $T$-ec  models
%axiomatize the model companion $T^*$ of $T$.

Since $T^*$ is the model companion of $T$,
$T^*_\forall=T_\forall$. In particular if $\psi\in T^*$, $\psi$ is consistent with 
the universal fragment of $T$.

Conversely assume $\psi+T_\forall$ is consistent. We must show that $\psi\in T^*$.

Since $T$ is complete, $T^*$ is also complete:
If $T_1\supseteq T^*$ and $T_2\supseteq T^*$, we get
that $(T_i)_\forall\supseteq (T^*)_\forall=T_\forall$ for both $i$.
Since $T$ is complete we actually get that  $(T_i)_\forall=T_\forall$ for both $i$.
Now $T_i\supseteq T^*$ entails that both $T_i$ are model complete (every universal 
formula is $T^*$-equivalent ---hence also
$T_i$-equivalent--- to an existential formula). This gives that $T_1$ is the model companion of $T_2$,
hence we can find $\mathcal{M}_1\sqsubseteq\mathcal{M}_2$ with $\mathcal{M}_i$ a model of $T_i$. 
Since they are both models of $T$,
model completeness of $T$ entails that $\mathcal{M}_1\prec \mathcal{M}_2$; therefore $T_1=T_2=T^*$.

We get that $T^*$ is also the model companion
of $\psi+T_\forall$ (since $T^*_\forall=T_\forall=(T_\forall+\psi)_\forall$ and $T^*$ is model complete). 
Hence some model $\mathcal{M}$ of $T^*$ is a substructure of some model $\mathcal{N}$
of $T_\forall+\psi$.

Since $\psi$ is $\Pi_2$ and holds in $\mathcal{N}$, $\psi$ holds in $\mathcal{M}$ as well:
assume $\psi$ is $\forall \vec{x}\exists\vec{y}\phi(\vec{x},\vec{y})$ with $\phi$ quantifier free.
Let $\vec{a}\in\mathcal{M}$. Then there exists $\vec{b}\in\mathcal{N}^{<\omega}$ such that
\[
\mathcal{N}\models \phi(\vec{a},\vec{b}).
\]
Since $\mathcal{M}\prec_1\mathcal{N}$
such  a $\vec{b}$ can be found in $\mathcal{M}^{<\omega}$. We conclude that
\[
\mathcal{M}\models
\forall \vec{x}\exists\vec{y}\phi(\vec{x},\vec{y}).
\]
Hence $\psi$ in $T^*$, being consistent with it. And this concludes the proof.
\end{proof}

In particular the Corollary outlines that forcibility exhausts all means to produce
the consistency of a $\Pi_2$-sentence for $\tau_{\UB}$ with the 
universal fragment of the $\tau_{\UB}$-theory of $V$.

%\section*{Some final remarks and future directions}

There are by now a variety of generic absoluteness results for the theory of $H_{\omega_2}$ assuming 
forcing axioms; for example the $\Pi_2$-maximality results for the theory of $H_{\aleph_2}$ assuming 
Woodin's axiom $(*)$ (see the monograph \cite{HSTLARSON}), or the second author's 
\cite{VIAASP,VIAAUD14,VIAMM+++,VIAMMREV,VIAUSAX} regarding the invariance of the first order theory of 
$H_{\aleph_2}$ under stationary set preserving forcings which preserve strong forcing axioms.

In follow-ups of this paper we will show that also these generic absoluteness results 
are strictly intertwined with model companionship (see \cite{VIATAMSTI,VIATAMSTII}).

\bibliographystyle{plain}
	\bibliography{Biblio}
\end{document}